\documentclass[11pt,reqno]{amsart}
\usepackage[top=3.5cm,bottom=3.5cm,left=2.5cm,right=2.5cm,heightrounded,bindingoffset=0mm]{geometry}

\usepackage[T1]{fontenc}
\usepackage[utf8]{inputenc}
\usepackage[english]{babel}
\usepackage{microtype}
\usepackage{lmodern}
\usepackage{amsmath,amsthm,amssymb,mathrsfs}

\usepackage{color}
\usepackage{graphicx}
\usepackage{tikz}

\newcommand{\bq}{\begin{equation}}
\newcommand{\eq}{\end{equation}}
\newcommand{\bqa}{\begin{eqnarray*}}
\newcommand{\eqa}{\end{eqnarray*}}

\theoremstyle{plain}
\newtheorem{theo}{Theorem}[section]
\newtheorem{prop}[theo]{Proposition}
\newtheorem{lemm}[theo]{Lemma}

\newtheorem{defi}[theo]{Definition}
\theoremstyle{definition}
\newtheorem{rema}[theo]{Remark} 

\DeclareMathOperator{\dist}{dist}

\DeclareMathOperator{\di}{div}

\DeclareSymbolFont{pletters}{OT1}{cmr}{m}{sl}
\DeclareMathSymbol{s}{\mathalpha}{pletters}{`s}

\def\tt{\theta}
\def\eps{\varepsilon}
\def\na{\nabla}

\def\mez{\frac{1}{2}}
\def\tdm{\frac{3}{2}}

\def\Rr{\mathbb{R}}
\def\Nn{\mathbb{N}}
\def\Zz{\mathbb{Z}}
\def\Cc{\mathbb{C}}

\def\cF{\mathcal{F}}

\def\L1{\mathcal{L}^{(1)}}
\def\L2{\mathcal{L}^{(2)}}
\def\L3{\mathcal{L}^{(3)}}

\def\p{\partial}

\def\na{\nabla}

\def\ka{\kappa}

\def\ol{\overline}
\def\T{\mathbb{T}}

\def\ka{\kappa}

\def\rH{\mathring{H}}
\def\rC{\mathring{C}}

\numberwithin{equation}{section}

\def\a{\alpha}

\usepackage{hyperref}
\hypersetup{colorlinks={true},linkcolor={blue},citecolor=blue}

\title{Large traveling capillary-gravity waves for Darcy flow}
\date{\today}

\author{Huy Q. Nguyen}
\address{
Department of Mathematics\\
University of Maryland\\
College Park, MD 20742, USA
}
\email[H. Q. Nguyen]{hnguye90@umd.edu}

\begin{document}

\begin{abstract}
We study surface capillary-gravity  waves for  viscous fluid flows governed by Darcy's law. This includes flows in vertical Hele-Shaw cells and in porous media (the one-phase Muskat problem) with finite or infinite depth. The free boundary is acted upon by an external pressure posited to be in traveling wave form with an arbitrary  periodic profile and an amplitude parameter. For any given  wave speed, we  first prove that there exists a unique local curve of small periodic traveling waves  corresponding to small values of the parameter. Then we prove that as the parameter increases but could possibly be bounded, the curve belongs to a connected set $\mathcal{C}$ of traveling waves. The  set $\mathcal{C}$ contains traveling waves that either have arbitrarily large gradients  or are arbitrarily close to the rigid bottom in the finite depth case. 
\end{abstract}

\keywords{Darcy law, Muskat problem, Hele-Shaw,  capillary,  traveling waves, surface waves}

\thanks{\em{MSC Classification: 35Q35, 76D27, 76D03, 76D45}}

\maketitle

\section{Introduction}
Traveling  surface waves for inviscid fluids  is a classic subject in mathematical fluid mechanics. After the pioneering work of Stokes \cite{Stokes}, small periodic traveling  waves were rigorously constructed by Nekrasov \cite{Nekrasov} and Levi-Civita \cite{Levi} in the infinite depth case and by Struik \cite{Struik} in the finite depth case. Subsequent developments  have led to constructions of large traveling waves including the extreme Stokes waves with angle $2\pi/3$ \cite{Krasovskii, Toland, AFT, Plotnikov}, solitary waves \cite{AmickToland}, and periodic traveling waves with vorticity \cite{ConstantinStrauss}. 
A review of the vast literature on this subject is beyond the scope of this paper and we refer to \cite{Ebbandflow} for a recent survey.  On the contrary, the mathematical theory of traveling surface waves for viscous fluids is far less developed. Since viscous fluids dissipate energy,  traveling surface waves can only exist when the system is appropriately forced. In the experimental works \cite{DCDA, MasnadiDuncan, ParkCho1}, a tube blowing air onto the surface of a viscous fluid is uniformly translated above the surface, resulting in the observation of traveling surface waves. Motivated in part by these experimental works, Leoni-Tice \cite{LeoniTice} constructed for the first time traveling waves for the free boundary gravity and capillary-gravity incompressible Navier-Stokes equations forced by a bulk force and an external stress tensor on the free boundary. They are small waves in suitable function spaces and their profiles are asymptotically flat at  infinity. The construction in \cite{LeoniTice} was then extended in \cite{StevensonTice1} to the multiplayer configuration, and in \cite{KoganemaruTice} to allow for traveling waves that are periodic in certain directions or all directions or traveling waves above an inclined flat bed. In \cite{StevensonTice2}, Stevenson-Tice constructed small asymptotically-flat traveling waves to the free boundary compressible Navier-Stokes equations. On the other hand, for viscous  flow in porous media, Nguyen-Tice \cite{NguyenTice}  constructed small traveling waves to the one-phase Muskat problem forced by a bulk force and an external pressure on the free boundary.  The waves constructed in \cite{NguyenTice} are either periodic or asymptotically flat, and it was proven that small periodic waves induced by external pressure on  free boundary are asymptotically stable. 

A common feature of all the above viscous free boundary problems is that they admit the equilibrium configuration of flat free boundary and quiescent fluid when there is no additional external force. The aforementioned constructions of traveling waves are perturbative from the equilibrium: small traveling waves are generated by small external forces.  Here, the smallness is imposed at least on the $C^{1, \a}$ norm of the free boundary. The main purpose of the present paper is to provide  examples of  viscous free boundary problems for which large traveling waves can be constructed non-perturbatively by means of  global continuation theory.  More precisely, we consider  capillary-gravity and capillary waves for fluid flows governed by Darcy's law
\bq\label{Darcy}
\frac{\mu}{\iota}u+\na_{x, y} p=-g\rho\vec{e_y},\quad \di u=0,
\eq
where $\vec{e_y}=(0, 1)\in \Rr^d\times \Rr$ is the vertical unit vector, and the physical dimensions are $d=1, 2$.  Here, $(x, y)\in \Rr^d\times \Rr$ denotes a point in the fluid domain, and $u$ and $p$ are respectively the fluid velocity and pressure. The positive parameters $\mu$, $g$ and $\rho$ are respectively the dynamic viscosity, acceleration due to gravity and fluid density. For flows in porous media, $\iota$ is the permeability constant, and the free boundary problem is called the one-phase Muskat problem. See Chapter 4 in \cite{Bear}. On the other hand, for two-dimensional flows in a vertical Hele-Shaw cell, lubrication theory gives $\iota=h^2/(12\mu)$, where  $h$ is the distance between the vertical plates. See Section 11.4 in \cite{Bear}. For notational simplicity we set
\[
\frac{\mu}{\iota}=\rho=1
\]
throughout this paper. 

For the sake of constructing periodic traveling waves we posit that the free boundary is the graph of a periodic function $\eta(x, t): \T^d\times \Rr_+\to \Rr$, where $\T^d=\Rr^d\setminus (2\pi \Zz)^d$ is the $d$-dimensional torus. This choice of a square torus is only for convenience and the sides of the torus can be arbitrary. The fluid domain is either
\bq\label{domain:i}
\Omega_\eta=\{(x, y)\in \T^d\times \Rr: y<\eta(x, t)\}
\eq
for infinite depth (deep fluid) or
\bq\label{domain:f}
\Omega_\eta=\{(x, y)\in \T^d\times \Rr: -b<y<\eta(x, t)\}
\eq
for finite depth above a flat bottom $\{y=-b\}$, where $b>0$ is a constant. To unify notation for the free boundary and the bottom, we denote the graph of a function $f:\T^d\to \Rr$ by
\bq
\Sigma_f=\{(x, f(x)): x\in\T^d\}.
\eq
In the finite depth case, the fluid velocity is tangent to the bottom, 
\bq
u_y(x, -b)=0,\quad\text{where } u_y:=u\cdot e_y.
\eq
Motivated by the aforementioned  experimental and theoretical works on viscous surface waves, we consider an external  pressure $\Psi(x, y, t)$ applying to the free boundary. For capillary-gravity  waves, the effect of surface tension is taken into account, and hence the dynamic boundary condition for the pressure is 
\bq\label{bc:p}
p=p_0+\sigma H(\eta)+\Psi \quad\text{on } \Sigma_{\eta},
\eq
where $p_0$ denotes the constant pressure of the region above $\Omega_\eta$, $\sigma>0$ is the surface tension coefficient, and 
\bq
H(\eta)=-\di \left(\frac{\na \eta}{\sqrt{1+|\na \eta|^2}}\right)
\eq
is twice the mean curvature of the free boundary. Finally, the free boundary moves with the fluid according to the kinematic boundary condition
\bq\label{kinematic}
\p_t \eta=u\cdot N\vert_{\Sigma_\eta},\quad N=(-\na \eta, 1).
\eq
Towards the construction of traveling waves we posit that the external pressure has the traveling wave form 
\bq\label{externalpressure}
\Psi(x, y, t)=\ka \varphi(x- \gamma \vec{e_1}t),
\eq
where $\ka \in \Rr$ is the amplitude parameter and $\gamma$ is the wave speed. This form of the external pressure is  to prepare for the involvement of global continuation theory. We have chosen  without loss of generality that the direction of wave propagation is $\vec{e_1}$. We have also chosen $\Psi$ to be independent of $y$, and we will show that this  suffices to furnish large traveling waves.  Then we accordingly make the traveling wave ansatz for the free boundary (by abuse of notation)
\bq
\eta(x, t)=\eta(x- \gamma \vec{e_1}t).
\eq
Introducing the Dirichlet-to-Neumann operator $G[\eta]$ for $\Omega_\eta$ (see Definition \ref{defi:DN}), the free boundary problem described above can be reformulated in the compact form 
\bq\label{eq:tw}
-\gamma \p_1\eta=-G[\eta](\sigma H(\eta)+g\eta+\ka \varphi),
\eq
where $\p_1=\p_{x_1}$. In the finite depth case, \eqref{eq:tw} must be coupled with the condition 
\bq\label{sep:cond}
\inf_{\T^d}(\eta+ b)>0.
\eq
See Section \ref{section:DN} for the derivation of \eqref{eq:tw}. We note that the constant pressure $p_0$ does not appear in \eqref{eq:tw} because $G[\eta]$ annihilates constants. 

Denoting by $\rC^{k, \a}(\T^d)$ the space of $C^{k, \a}$ functions on $\T^d$ with mean zero, our main result is  as follows.
\begin{theo}\label{theo:main}
We consider $\sigma>0$ and $g>0$, and assume that the fluid domain is either \eqref{domain:i} (infinite depth) or \eqref{domain:f} (finite depth)  with the separation condition \eqref{sep:cond}. Let $d\in \{1, 2\}$  and $\alpha \in (0, 1)$. Fix an arbitrary external pressure profile $\varphi \in \rC^{1, \a}(\T^d)\setminus\{0\}$ and an arbitrary wave speed $\gamma \in \Rr\setminus\{0\}$. The following assertions hold.

I. (\textbf{Small waves}) Near the trivial solution $(\eta, \ka)=(0, 0)$, \eqref{eq:tw} has a unique curve $\mathcal{C}_0$ of solutions $(\eta, \ka) \in \rC^{3, \a}(\T^d)\times \Rr$ parametrized by $|\ka|<\eps=\eps(\sigma, g, \gamma, d, b,  \a)$. Moreover, the curve is Lipschitz continuous.

II.  (\textbf{Large waves}) Let $\mathcal{C}$ denote the connected component of the set of all solutions $(\eta, \ka)\in \rC^{3, \a}(\T^d)\times \Rr$ of \eqref{eq:tw} to which the local curve $\mathcal{C}_0$ belongs. Then the following  holds.
\begin{itemize}
\item In the infinite depth case, $\mathcal{C}$ contains traveling waves $\eta$  that are arbitrarily large in $C^1(\T^d)$.
\item In the finite depth case, $\mathcal{C}$ contains traveling waves $\eta$ that are either  arbitrarily large in $C^1(\T^d)$  or arbitrarily close to the bottom. Precisely, the latter means that there exist  waves $\eta_n$ satisfying 
\bq\label{touch}
\lim_{n\to \infty}\inf_{x\in \T^d}(\eta_n(x)+b)=0.
\eq 
\end{itemize}
III. (\textbf{Regularity}) For any $k\ge 3$ and $\mu\in (0, 1)$, if $ \varphi \in \rC^{k-2, \mu}(\T^d)$ then $\mathcal{C}\subset \rC^{k, \mu}(\T^d)\times \Rr$. 
\end{theo}
Assuming that \eqref{touch} holds, then since $\eta_n$ has mean zero, the mean value theorem implies that 
\bq
\liminf_{n\to \infty}\| \na\eta_n\|_{C(\T^d)}\ge \frac{b}{2\pi\sqrt{d}}.
\eq
Therefore, Part II of Theorem \ref{theo:main} implies that   there are traveling waves whose maximum gradients are unboundedly large in the infinite depth case and larger than $\frac{b}{4\pi \sqrt{d}}$ in the finite depth case. To the best of our knowledge, this is the first construction  of large traveling surface waves for a viscous free boundary problem. Some remarks about the statement of Theorem \ref{theo:main} are in order:
\begin{itemize}
\item[1)] If the pair $(\eta, \varphi)$ satisfies \eqref{eq:tw} then so does $(\eta, \varphi+c)$ for any constant $c$ since $G[\eta]c=0$. Therefore, we can assume without loss of generality that $\varphi$ has mean zero.
\item[2)] For $\eta\in \rC^{3, \a}$ and $\varphi \in \rC^{1, \a}$, equation \eqref{eq:tw} is satisfied in $C^{0, \a}$, hence $\eta$ is a classical solution. 
\item[3)] Part I will be proven in Proposition \ref{prop:smallTW} for the more general form $\Psi(x, y, t)=\psi(x-\gamma \vec{e_1}t)$ of the external pressure.
\item[4)] The proof of part II does not exclude the possibility that traveling waves with unboundedly large gradients correspond to a bounded  set of amplitudes of the external pressure. 
\item[5)] We will prove that Theorem \ref{theo:main} remains valid in dimension $d\ge 3$ but with the $C^1$ norm in part II replaced by $C^{1, \beta}$ for any $\beta\in (0, 1)$. 
\end{itemize}
\begin{rema}
Following the method in \cite{NguyenTice} for gravity waves,  the unique local curve of traveling capillary-gravity waves near $(\eta, \ka)=(0, 0)$ can be obtained for $\eta\in \mathcal{H}^s(\Rr^d)$, the anisotropic Sobolev space introduced in \cite{LeoniTice}. We do not pursue this issue since the main purpose of the present paper is to construct large traveling waves. 
\end{rema}
\begin{rema}
In Proposition \ref{prop:stlim} we prove that if $g>0$ and $\varphi\in \rC^{3, \a}(\T^d)$ then the smallness $\eps$ of the local curve in part I of Theorem \ref{theo:main} can be chosen uniformly in $\sigma \in (0, 1)$. This in turn allows us to establish the vanishing surface tension limit of the traveling capillary-gravity waves, where the limit is the unique traveling gravity wave. 
\end{rema}
\begin{rema}
In the absence of external pressure, i.e. $\Psi=0$ in \eqref{bc:p}, well-posedness of the dynamic Muskat problem \eqref{Darcy}-\eqref{kinematic} was established in \cite{Ng-st}. The method in \cite{Ng-st} can be adapted to prove well-posedness when $\Psi\ne 0$. We also refer to \cite{Amb, CorCorGan, CorCorGan2, CheGraShk, NgPa, AN, DGN, DGN2, APW, GGSP} and the references therein for results on the Muskat problem without surface tension ($\sigma =0$ and $g>0$). 
\end{rema}
We  now discuss the main ideas of the proof of Theorem \ref{theo:main}. In the class of classical $\rC^{3, \a}(\T^d)$ solutions,  we shall use two different reformulations of \eqref{eq:tw} to prove part I and part II  of Theorem \ref{theo:main}. Regarding the first part, we shall linearize the Dirichlet-to-Neumann  operator $G[\eta]$ and the mean curvature operator $H$ about zero,
\bq
G[\eta]f=m(D)f+R[\eta]f,\quad H(\eta)=-\Delta \eta+R_H(\eta),
\eq
where $m(D)=G[0]$ is the  Fourier multiplier
\[
m(\xi)=
\begin{cases}
|\xi|\tanh(b|\xi|)\quad\text{for finite depth},\\
|\xi|\quad\text{for infinite depth}.
\end{cases}
\]
Then we can reformulate the traveling wave equation \eqref{eq:tw} as the fixed point problem 
\bq
\begin{aligned}
 \eta&=K^{\sigma, g}_\psi(\eta)\\
 &:=[-\gamma\p_1+m(D)(-\sigma \Delta +g)]^{-1}\left\{-\sigma m(D)R_H(\eta)-R[\eta](\sigma H(\eta)+g\eta)- G[\eta]\psi\right\}
\end{aligned}
\eq
provided $\eta$ has mean zero. In order to prove that $K^{\sigma, g}_\psi$ is a contraction on a small ball in $\rC^{3, \a}(\T^d)$ for $\sigma>0$ and $g\ge 0$, we shall establish necessary  boundedness and contraction estimates for the remainders $R[\eta]$ and $R_H$  in H\"older spaces. See Propositions \ref{prop:linDN} and \ref{prop:linH}. These estimates actually hold for large $\eta$, and hence will be  useful in the construction of large traveling waves in Part II.

As for the proof of part II, the key tool is  a Global Implicit Function Theorem (GIFT) based upon the Leray-Schauder degree theory. This tool has recently been employed in \cite{StraussWu, StraussWu2, StraussWu3} to construct rapidly rotating solutions for models of stars and galaxies. An important condition of this GIFT is that the equation is a compact perturbation of identity. For our equation \eqref{eq:tw} we shall achieve this by establishing the invertibility in H\"older spaces of the capillarity-gravity operator $\sigma H+gI$ and the Dirichlet-to-Neumann operator $G[\eta]$ for any (large) $\eta$. See Propositions \ref{theo:iso:G} and \ref{theo:iso:H}. This allows us to obtain the second reformulation of  \eqref{eq:tw}, namely,
\bq
\cF(\eta, \ka):=\eta+F(\eta, \ka)=0,
\eq
where 
\bq
F(\eta, \ka)= - (\sigma H+g I)^{-1}(G[\eta])^{-1} (\gamma\p_1\eta)+  (\sigma H+g I)^{-1}(\ka \varphi)
\eq
 is a compact operator on $\rC^{3, \a}(\T^d)$ provided  $\varphi\in \rC^{1, \a}(\T^d)$. The compactness of $F$ essentially comes from the fact that the operator $(G[\eta])^{-1} \p_1$ is of order zero and the operator $(\sigma H+g I)^{-1}$ is of order $-2$. Although it is not clear if the mapping $\eta\mapsto (G[\eta])^{-1}$ is $C^1$ (in the Fr\'etchet sense) near $\eta=0$, we shall  prove that $(G[\eta])^{-1}\p_1\eta$ is Fr\'etchet differentiable at $\eta=0$ and so is $F$. Moreover, we have that
\[
D_\eta\cF(0, 0)=I-\gamma(-\sigma \Delta+gI)^{-1}m^{-1}(D)\p_1
\]
is invertible on $\rC^{3, \a}(\T^d)$. Fortunately, these conditions suffice to apply the version of the GIFT stated in Theorem \ref{GIFT}. It yields a connected set $\mathcal{C}\subset \rC^{3, \a}(\T^d)\times \Rr$ of solutions to the equation $\cF(\eta, \ka)=0$ to which the local curve in part (1) belongs. In the infinite depth case, we obtain two alternatives: either (i)  $\mathcal{C}$ is unbounded, or (ii) $\mathcal{C}\setminus\{(0, 0)\}$ is connected. The  second alternative is of the least interest to us because it could mean that $\mathcal{C}$ is a loop which does not contain unboundedly large solutions. To exclude this scenario, we shall prove that $\eta=0$ is the unique solution when $\ka=0$ (i.e. no external pressure). {\it This is a common feature in free boundary viscous problems: without additional external forces, the only traveling wave solution is the trivial one.  On the other hand, it is this feature which precludes the "loop" scenario in the GIFT.} We conclude that  (i) holds, i.e. $\mathcal{C}$ is unbounded. If $\ka$ is bounded in $\mathcal{C}$, then $\eta$ must be unbounded in $\mathcal{C}$. On the other hand, if $\ka$ is unbounded in $\mathcal{C}$, we need  to prove that $\eta$ is again unbounded $\mathcal{C}$.  We shall first prove that in both cases the $C^{1, \beta}$ norm of $\eta$ is unbounded in $\mathcal{C}$ for any $\beta\in (0, 1)$ and $d\ge 1$. For the physical dimensions $d=1, 2$, in order to obtain the unboundedness of the maximum gradient, i.e. the $C^1$ norm, of $\eta$ in $\mathcal{C}$, we appeal to the optimal solvability of the Neumann problem in Lipschitz domains \cite{DahlbergKenig}.  In the finite depth case, an additional alternative is that $\mathcal{C}$ contains a sequence of waves approaching the bottom. 

 The remainder of this paper is organized as follow. Sections \ref{section:DN} and \ref{section:H} are devoted to the invertibility, linearization and contraction of the Dirichlet-to-Neumann operator and the capillary-gravity operator $\sigma H(\eta)+gI$. Theorem \ref{theo:main} is proven in Section \ref{section:proof}. Finally, some facts about Fourier multipliers in H\"older spaces are recorded in Appendix \ref{appendix}.
\section{The Dirichlet-to-Neumann operator}\label{section:DN}
We first set notation for Sobolev spaces of distributions with mean zero. \begin{defi}
For $s\ge 0$, we set 
\begin{align}
&\rH^s(\T^d)=\left\{f\in H^s(\T^d): \int_{\T^d} f=0\right\},\\ \label{def:rH:neg}
&\rH^{-s}(\T^d)=\left\{f\in H^{-s}(\T^d): \langle f, 1\rangle_{H^{-s}, H^s}=0\right\},
\end{align}
where $H^{-s}(\T^d)$ denotes the dual of $H^s(\T^d)$. 
\end{defi}
\subsection{Invertibility}
\begin{defi}\label{defi:DN}
The Dirichlet-to-Neumann operator $G[\eta]$ associated to the domain $\Omega_\eta$ is defined by 
\bq
G[\eta]f= \na q\cdot N\vert_{\Sigma_\eta}
\eq
where $N=(-\na \eta, 1)$ is normal to $\Sigma_f$, and in the finite depth case $q$ is the solution to the problem
\bq\label{elliptic:DN}
\begin{cases}
\Delta q=0\quad\text{in } \Omega_\eta,\\
q\vert_{\Sigma_\eta}=f,\\
\p_yq\vert_{\Sigma_{-b}}=0,
\end{cases}
\eq 
while in the the infinite depth case $q$ is the solution to the problem 
\bq\label{elliptic:DN:i}
\begin{cases}
\Delta q=0\quad\text{in } \Omega_\eta,\\
q\vert_{\Sigma_\eta}=f,\\
 \na q\in L^2(\Omega_\eta).
\end{cases}
\eq
\end{defi}
We note that  \eqref{elliptic:DN} and \eqref{elliptic:DN:i} are invariant under the transformation  $(q, c)\mapsto  (q+c, f+c)$ for any $c\in \Rr$. In addition, $\na q$ is also invariant under $q\mapsto q+c$. This observation allows us to restrict the domain of the Dirichlet-to-Neumann operator to functions of mean zero.

Assume that $\eta \in W^{1, \infty}(\T^d)$. In the finite depth case, \eqref{elliptic:DN} has a unique solution in $H^1(\Omega_\eta)$ for any $f\in H^\mez(\T^d)$. In the infinite depth case, using the continuity and lifting results for the trace operator $\dot H^1(\Omega_\eta)\to \dot H^\mez(\T^d)$ (see Theorems 18.27 and 18.28 in \cite{Leoni}), it follows that \eqref{elliptic:DN:i} has a unique solution in $\dot H^1(\Omega_\eta)$ for any $f\in \dot H^\mez(\T^d)$. Here, for any open subset $U\subset \Rr^n$ we define
\bq
\dot H^1(U)=\{ u\in L^2_{loc}(U): \na u\in L^2(U)\}/\Rr.
\eq
We refer to Propositions 3.4 and 3.6 in \cite{NgPa} for the proofs. 

With the Dirichlet-to-Neumann operator defined above, we  now derive the traveling wave equation \eqref{eq:tw}. In Darcy's law $u+\na_{x, y}p=-g \vec{e_y}$, setting $q=p+g y$ we have $u=-\na_{x, y}q$. Using the dynamic condition \eqref{bc:p} with $\Psi=\Psi(x, t)$, we find 
\[
q\vert_{\Sigma_\eta}=p_0+\sigma H(\eta)+\Psi+g\eta,
\]
and thus 
\[
 \na_{x, y}q\cdot N\vert_{\Sigma_\eta}=G[\eta](p_0+\sigma H(\eta)+\Psi+g\eta)=G[\eta](\sigma H(\eta)+g\eta+\Psi).
\]
Then the kinematic boundary condition \eqref{kinematic} yields
\bq\label{dynamic:eta}
\p_t\eta=u\cdot N\vert_{\Sigma_\eta}= -\na_{x, y}q\cdot N\vert_{\Sigma_\eta}=-G[\eta](\sigma H(\eta)+g\eta+\Psi).
\eq
Assuming that $\Psi(x, t)=\ka \varphi(x- \gamma \vec{e_1}t)$ as in \eqref{externalpressure} and imposing the traveling wave ansatz $\eta(x, t)=\eta(x- \gamma \vec{e_1}t)$ (by abuse of notation), we arrive at the traveling wave equation \eqref{eq:tw} for the profile $\eta$.

Next, we recall the following Stokes formula (see Section 3.2, Chapter IV, \cite{BoyerFabrie}). 
\begin{prop}
Let $U\subset \Rr^n$ be Lipschitz domain with compact boundary and denote 
\bq
H_{\di}(U)=\{u\in L^2(U)^n: \di u\in L^2(U)\}.
\eq
 If $u\in H_{\di}(U)$ and $w\in H^1(U)$, then 
\bq\label{Stokes}
\int_U u\cdot \na w+\int_U w\di u=\langle \gamma_\nu(u), \gamma_0(w)\rangle_{H^{-\mez}(\p U), H^\mez(\p U)}
\eq
where $\gamma_0(w)$ is the trace of $w$ and $\gamma_\nu (u)$ is the trace of  $u\cdot \nu$, $\nu$ being the unit outward normal to $\p U$. The trace operator 
\bq\label{normaltrace}
\gamma_\nu:\quad H_{\di}(U)\to H^{-\mez}(\p U)
\eq
is continuous.
\end{prop}
For Propositions \ref{DN:positive} and \ref{theo:iso:G} we assume in the finite depth case that 
\bq
\inf_{x\in \T^d}(\eta(x)+b)\ge c^0>0.
\eq
\begin{prop}\label{DN:positive}
If $\eta \in W^{1, \infty}(\T^d)$ then $G[\eta]: \rH^\mez(\T^d)\to \rH^{-\mez}(\T^d)$ is an isomorphism and is positive-definite. Moreover, we have
\bq\label{inverseDN:low}
\| (G[\eta])^{-1}\|_{ \rH^{-\mez}\to  \rH^\mez}\le M(\| \eta\|_{W^{1, \infty}})
\eq
for some $M:\Rr_+\to \Rr_+$ depending only on $(d, b, c^0)$.
\end{prop}
\begin{proof}
We first note that $G[\eta]$ maps $H^\mez(\T^d)$ continuously to $H^{-\mez}(\T^d)$ by virtue of the trace \eqref{normaltrace}. Indeed, for both the finite and infinite depth cases we can apply  \eqref{normaltrace} with $u=\na q\in H_{\di}(\Omega_f)$ since $\na q\in L^2(\Omega_f)$ and $\Delta q=0$.

{\it The finite depth case.}  Let $f\in \rH^\mez(\T^d)$ and let $q\in H^1(\Omega_f)$ be the  unique solution of \eqref{elliptic:DN}.  An application of the Stokes formula \eqref{Stokes} with $u=\na q$ and $w=1$ yields $ \langle G[\eta]f, 1\rangle_{H^{-\mez}, H^\mez}=0$, hence $G[\eta]$ maps $H^\mez(\T^d)$ to $\rH^{-\mez}(\T^d)$. On the other hand, applying \eqref{Stokes} with $u=\na q$ and $w=q$, we find
\bq
\langle G[\eta]f, f\rangle_{H^{-\mez}, H^\mez}=\int_{\Omega_f}|\na q|^2\ge 0.
\eq 
Assume now that $\langle G[\eta]f, f\rangle_{H^{-\mez}, H^\mez}=0$. It follows that  $q$ is constant and so is $f=q\vert_{\Sigma_\eta}$, hence $f=0$.  Thus $G[\eta]$ is positive-definite.

{\it The infinite depth case.} Let $f\in \rH^\mez(\T^d)$ and let $q\in \dot H^1(\Omega_\eta)$ be the unique solution of \eqref{elliptic:DN:i}. To obtain $ \langle G[\eta]f, 1\rangle_{H^{-\mez}, H^\mez}=0$ we cannot apply  \eqref{Stokes} with $w=1\notin L^2(\Omega_\eta)$. This can be remedied by approximating $1$ by $w(x, y)=\chi(y/n)$ where $\chi=1$ on $(-1, \infty)$ and $\chi$ vanishes on $(-\infty, -2)$. Similarly, to obtain the positivity of $G[\eta]$ we apply \eqref{Stokes} with $w(x, y)=q(x, y)\chi(y/n)$ and let $n\to \infty$.  

Since $G[\eta]$ is positive-definite, it is injective. The surjectivity of $G[\eta]$ and the norm estimate \eqref{inverseDN:low} for the inverse  will be proven in the proof of Proposition \ref{theo:iso:G} below.
\end{proof}
\begin{prop}\label{theo:iso:G}
For any $k\ge 1$ and $\a\in (0, 1)$, if $\eta\in C^{k, \a}(\T^d)$ then $G[\eta]: \rC^{k, \a}(\T^d)\to \rC^{k-1, \a}(\T^d)$ is an isomorphism and 
\bq\label{inverseDN}
\| (G[\eta])^{-1}\|_{\rC^{k-1, \a}\to \rC^{k, \a}}\le M(\| \eta\|_{C^{k, \a}})
\eq
for some $M:\Rr_+\to \Rr_+$ depending only on $(d, b, c^0, k, \a)$. 
\end{prop}
\begin{proof}
By  Schauder's estimates,  $G[\eta]$ maps $C^{k, \a}(\T^d)$ to  $C^{k-1, \a}(\T^d)$ continuously. Recalling from the proof of Proposition \ref{DN:positive} that $G[\eta]: \rH^\mez(\T^d)\to \rH^{-\mez}(\T^d)$ is injective, we deduce that $G[\eta]: \rC^{k, \a}(\T^d)\to \rC^{k-1, \a}(\T^d)$ is also injective. 

To prove the surjectivity of $G[\eta]$ in the finite depth case, we note that for any $h\in \rH^\mez\T^d)$, the Neumann problem 
\bq\label{eq:Neumann}
\begin{cases}
\Delta q=0\quad\text{in } \Omega_\eta,\\
\p_Nq\vert_{\Sigma_\eta}=h,\\
\p_yq\vert_{\Sigma_{-b}}=0
\end{cases}
\eq
has a unique solution $q\in H^1(\Omega_\eta)$ with mean zero and 
\[
\| q\|_{H^1(\Omega_\eta)}\le M(\| \eta\|_{W^{1, \infty}})\| h\|_{H^{-\mez}},
\]
where $M$ depends only on $(d, b, c^0)$. Then we set $f=q\vert_{\Sigma_\eta}\in  H^\mez(\T^d)$, so that $G[\eta]f=h$.  As remarked after Definition \ref{defi:DN},   without changing $G[\eta]f$ we can subtract the mean of $f$ so that $f\in \rH^\mez(\T^d)$. Therefore $G[\eta]: \rH^{\mez}(\T^d)\to \rH^{-\mez}(\T^d)$ is bijective and its inverse is bounded as in \eqref{inverseDN:low}.  If $h\in C^{k-1, \a}(\T^d)$ with $k\ge 1$, then classical elliptic regularity gives $q\in C^{k, \a}(\ol{\Omega_\eta)}$ and 
\[
\| q\|_{C^{k, \a}(\ol{\Omega_\eta)}}\le M(\| \eta\|_{C^{k, \a}})\| h\|_{C^{k-1, \a}},
\]
where $M$ depends only on $(d, b, c^0, k, \a)$. See Section 6.7 in \cite{GilTru}.  Consequently $f=(G[\eta])^{-1}h\in C^{k, \a}(\T^d)$, and thus  $G[\eta]: \rC^{k, \a}(\T^d)\to \rC^{k-1, \a}(\T^d)$ is bijective and its inverse satisfies \eqref{inverseDN}.

Next we consider the infinite depth case. Since the domain is unbounded we do not have the same solvability of the Neumann problem as in the finite depth case.  For $h\in  \rH^{-\mez}(\T^d)$ we consider the Neumann problem
\bq
\begin{cases}
\Delta q=0\quad\text{in } \Omega_\eta,\\
\p_Nq\vert_{\Sigma_\eta}=h,\\
 \na q\in L^2(\Omega_f).
\end{cases}
\eq
Assuming  that $q$ is a smooth solution, then for any test function $\phi$ vanishing at large depth, we obtain after integration by parts that 
\bq\label{variational:Neumann}
\int_{\Omega_\eta}\na q\cdot \na\phi =\langle h, \phi\vert_{\Sigma_\eta}\rangle_{H^{-\mez}(\T^d), H^\mez(\T^d)}=\langle h, \phi\vert_{\Sigma_\eta}\rangle_{H^{-\mez}(\T^d), \rH^\mez(\T^d)},
\eq
where we have used the fact that $\langle h, 1\rangle_{H^{-\mez}, H^\mez}=0$.  For the infinite depth domain $\Omega_f$, the trace operator is continuous from $\dot H^1(\Omega_\eta)$ to $\dot H^\mez(\T^d)$ and we identify $\dot H^\mez(\T^d)$ with $\rH^\mez(\T^d)$ by choosing the representative in each equivalence class to have mean zero.  Then the variational formulation \eqref{variational:Neumann} has a unique solution $q\in \dot H^1(\Omega_\eta)$ with 
\[
\| q\|_{\dot H^1(\Omega_\eta)}\le M(\|\eta \|_{W^{1, \infty}})\| h\|_{H^{-\mez}}.
\]
 Defining $f=q\vert_{\Sigma_\eta}\in \rH^\mez(\T^d)$ we obtain that $G[\eta]f=h$. This means that $G[\eta]$ maps $\rH^\mez(\T^d)$ onto $\rH^{-\mez}(\T^d)$. Now if  $h\in \rC^{k-1, \a}$ with $k\ge 1$, then classical elliptic regularity implies that $q\in C^{k, \a}$ on the closure of any bounded subset of $\Omega_{\eta}$, hence $f\in \rC^{k, \a}(\T^d)$. This completes the proof of the proposition.
\end{proof}
\subsection{Linearization and contraction}
\begin{defi}
For $a:\Rr^d\to \Cc$, the Fourier multiplier $a(D)$ on $\T^d$ is defined by
\bq
\widehat{a(D)f}(k)=a(k)\hat{f}(k),\quad k\in \Zz^d,
\eq
where the Fourier transform on $\T^d$ is given by 
\[
\hat{f}(k)=\int_{\T^d}f(x)e^{-ik\cdot x}dx,\quad k\in \Zz^d.
\]
\end{defi}
\begin{prop}\label{prop:linDN}
Denote 
\bq\label{def:m}
m(\xi)=
\begin{cases}
|\xi|\tanh(b|\xi|)\quad\text{for finite depth},\\
|\xi|\quad\text{for infinite depth}.
\end{cases}
\eq
Let $k\ge 1$ and $\a\in (0, 1)$. Let   $\eta \in \rC^{k, \a}(\T^d)$ and assume   that  $\inf_{x\in \T^d} (\eta(x)+b)\ge c^0>0$ in the finite depth case. There exists $M: \Rr_+\to \Rr_+$ depending only on $(d, b, c^0,  k, \a)$ such that for any $f\in \rC^{k, \a}(\T^d)$ we have
\bq\label{lin:DN}
G[\eta]f=m(D)f+R[\eta]f,
\eq
where $R[\eta]f$ has mean zero and  
\bq\label{est:RDN}
\| R[\eta]f\|_{C^{k-1, \a}(\T^d)}\le M(\| \eta\|_{C^{k, \a}(\T^d)})\| \eta\|_{C^{k, \a}(\T^d)}\| f\|_{C^{k, \a}(\T^d)}.
\eq
Moreover, if  $\inf_{x\in \T^d} (\eta_j(x)+b)\ge c^0_j>0$ in the finite depth case, then there exists $\tilde M: \Rr_+\times \Rr_+\to \Rr_+$ depending only on $(d, b, c^0_1, c^0_2,  k, \a)$ such that 
\bq\label{contraction:RDN}
\| R[\eta_1]f-R[\eta_2]f\|_{C^{k-1, \a}(\T^d)}\le \tilde M\left(\| \eta_1\|_{C^{k, \a}(\T^d)}, \| \eta_2\|_{C^{k, \a}(\T^d)}\right)\| \eta_1-\eta_2\|_{C^{k, \a}(\T^d)}\| f\|_{C^{k, \a}(\T^d)}.
\eq
\end{prop}
\begin{proof}
We first note that since $G[\eta]f$ has mean zero and the symbol $m$ vanishes at zero frequency, the remainder $R[\eta]f$ defined by \eqref{lin:DN} also has mean zero. To prove the estimate \eqref{est:RDN} we  flatten the domain $\Omega_\eta$ using the change of variables $(x, z)\to (x, y)=(x, \rho(x, z))\in \Omega_\eta$ where 
\bq
\rho(x, z)=
\begin{cases}
\frac{z+b}{b}\eta(x)+z,\quad (x, z)\in \T^d\times (-b, 0)\quad\text{for  finite depth},\\
\eta(x)+z,\quad (x, z)\in \T^d\times \Rr_-\quad\text{for infinite depth}.
\end{cases}
\eq
The Jacobian of this change of variables is $\p_z\rho$ which equals $1$ for infinite depth and $(\eta(x)+b)/b\ge c^0>0$ for finite depth thanks to the assumption on $\eta$. In either case, it is a Lipschitz diffeomorphism. To express the Dirichlet-to-Neumann operator $G[\eta]f$ we let $q$ be the solution of \eqref{elliptic:DN} or \eqref{elliptic:DN:i}. Then $v(x, z)=q(x, \rho(x, z))$ satisfies 
\bq\label{elliptic:v}
\di_{x, z} \mathcal{A}\na_{x, z}v=0
\eq
where
\bq
\mathcal{A}=\begin{bmatrix}
\p_z\rho I_d & -\na_x\rho\\
-(\na_x\rho)^T &\frac{1+|\na_x\rho|^2}{\p_z\rho}
\end{bmatrix},
\eq
$I_d$ denoting the identity $d\times d$ matrix. We split 
\[
\begin{aligned}
\mathcal{A}&=I+\mathcal{A}_r,\quad \mathcal{A}_r=\begin{bmatrix} 
(\p_z\rho -1) I_d & -\na_x\rho\\
-(\na_x\rho)^T &\frac{1+|\na_x\rho|^2}{\p_z\rho}-1
\end{bmatrix}
\end{aligned}
\]
so that each entry of $\mathcal{A}_r$ vanishes as $\eta=0$ and the estimate \bq\label{est:cAr}
\| \mathcal{A}_r\|_{C^{k-1, \a}(\overline\Omega_\eta)}\le M(\| \eta\|_{C^{k, \a}(\T^d)})\| \eta\|_{C^{k, \a}(\T^d)}
\eq
holds for some $M$ depending only on $(d, b, c^0, k, \a)$.  We then split the solution accordingly as $v=v_0+v_r$, where 
\bq
\Delta_{x, z} v_0=0\quad\text{in }\Omega_0,\quad v_0\vert_{z=0}=f,\quad \p_zv_0\vert_{z=-b}=0
\eq
and 
\bq
\Delta_{x, z} v_r=-\di \mathcal{A}_r\na_{x, z}v\quad\text{in }\Omega_0,\quad v_r\vert_{z=0}=0,\quad \p_zv_r\vert_{z=-b}=0
\eq
with obvious modifications for the infinite depth case. It is a classical calculation using Fourier transform in $x$ that $\p_zv_0\vert_{z=0}=m(D)f$. On the other hand,  for $k\ge 2$ Schauder's estimates for $v$ and $v_r$ yield 
\[
\begin{aligned}
\| v_r\|_{C^{k, \a}(\ol{\Omega_\eta})}&\le C\| \di(\mathcal{A}_r\na_{x, z}v)\|_{C^{k-2, \a}(\ol{\Omega_\eta})}\\
&\le M(\| \eta\|_{C^{k, \a}(\T^d)})\| \eta\|_{C^{k, \a}(\T^d)}\| \na_{x, z} v\|_{C^{k-1, \a}(\ol{\Omega_\eta})}\\
&\le M(\| \eta\|_{C^{k, \a}(\T^d)})\| \eta\|_{C^{k, \a}(\T^d)}\| f\|_{C^{k, \a}(\T^d)},
\end{aligned}
\]
where we have invoked \eqref{est:cAr}. For $k=1$, the $C^{1, \a}$ elliptic estimate (see Theorem 8.33 in \cite{GilTru}) implies
\[
\begin{aligned}
\| v_r\|_{C^{1, \a}(\ol{\Omega_\eta})}&\le C\| \mathcal{A}_r\na_{x, z}v\|_{C^{0, \a}(\ol{\Omega})}\\
&\le M(\| \eta\|_{C^{k, \a}(\T^d)})\| \eta\|_{C^{1, \a}(\T^d)}\| \na_{x, z} v\|_{C^{0, \a}(\ol{\Omega})}\\\
&\le M(\| \eta\|_{C^{k, \a}(\T^d)})\| \eta\|_{C^{1, \a}(\T^d)}\| f\|_{C^{1, \a}(\T^d)}.
\end{aligned}
\]
Thus for all $k\ge 1$ we obtain
\bq\label{est:vr}
\| \na_{x, z} v_r\vert_{z=0}\|_{C^{k-1, \a}(\T^d)}\le M(\| \eta\|_{C^{k, \a}(\T^d)})\| \eta\|_{C^{k, \a}(\T^d)}\| f\|_{C^{k, \a}(\T^d)}.
\eq
Next, using the chain rule and the definition of $\mathcal{A}_r$, we find
\bq
\begin{aligned}
G[\eta]f
&=\p_zv\vert_{z=0}+\left(\frac{1+|\na \rho|^2}{\p_z\rho}-1\right)\p_z v\vert_{z=0}-\na_x\rho\cdot \na_xv\vert_{z=0}\\
&=m(D)+\p_zv_r\vert_{z=0}+\mathcal{A}_re_{d+1}\cdot \na_{x, z}v,
\end{aligned}
\eq
where $e_{d+1}=(0, 1)$ with $0\in \Rr^d$.   Then appealing to \eqref{est:vr}, \eqref{est:cAr},  and Schauder estimates for $v$ we obtain \eqref{est:RDN}. The contraction estimate \eqref{contraction:RDN} can be obtained by an analogous argument upon taking the difference of the equations \eqref{elliptic:v} for the solutions corresponding to $f_1$ and $f_2$. 
\end{proof}
Since $G[\eta_1]-G[\eta_2]=R[\eta_1]-R[\eta_2]$, \eqref{contraction:RDN}  yields a contraction estimate for the Dirichlet-to-Neumann operator. 
\section{The capillary-gravity operator}\label{section:H}
We first prove the invertibility of the capillary-gravity operator $\sigma H+gI$ between H\"older spaces.
\begin{prop}\label{theo:iso:H}
Let $\sigma>0$, $g>0$, $k\ge 3$, and $\a\in (0, 1)$. Then  $\sigma H+g I: C^{k, \a}(\T^d)\to C^{k-2, \a}(\T^d)$ is bijective and its inverse is Lipschitz continuous on bounded subsets of $C^{k-2, \a}(\T^d)$, that is,
\bq\label{continverse:H}
\| (\sigma H+g I)^{-1}(h_1)-(\sigma H+g I)^{-1}(h_2)\|_{C^{k, \a}}\le M(\| h_1\|_{C^{k-2, \a}}, \| h_2\|_{C^{k-2, \a}})\| h_1-h_2\|_{C^{k-2, \a}}
\eq
for some $M: \Rr_+\times \Rr_+\to \Rr_+$ depending only on $(\sigma, g, d, k, \a)$.

Moreover, for all $h\in C^1(\T^d)$, the equation  $(\sigma H+gI)(f)=h$ has a unique solution  $f\in C^{2, \beta}(\T^d)$ for all  $\beta\in (0, 1)$. 
\end{prop}
\begin{proof}
Clearly  $\sigma H+g I: C^{k, \a}(\T^d)\to C^{k-2, \a}(\T^d)$. To prove the invertibility we fix any $h\in C^{k-2, \a}(\T^d)$ and prove that the equation 
\bq\label{cgeq}
\sigma H(f)+g f=h
\eq
has a unique solution  $f\in C^{k, \a}(\T^d)$.  We first show the uniqueness by supposing 
\[
(\sigma H+gI)f_j=h_j,
\]
 where $h_j\in C^{k-2, \a}$ and $f_j\in  C^{k, \a}(\T^d)$ for $j=1, 2$. Then setting $f=f_1-f_2$ and $H_0(z)=z(1+|z|^2)^{-\mez}$ for $z\in \Rr^d$,   we have 
\[
-\sigma \di[(H_0(f_1))-H_0(f_2)]+gf=h_1-h_2.
\]
 Multiplying the preceding equation by $f$ and integrating by parts, we find
\begin{align*}
\sigma\int_{\T^d} \na f\cdot [(H_0(\na f_1)-H_0(\na f_2)]dx+g\int_{\T^d} f^2dx=\int_{\T^d}(h_1-h_2)fdx.
\end{align*}
For any fixed $x$ we define the scalar function $\ell_x: \Rr^d\to \Rr$ by $\ell_x(p)=\na f(x)\cdot H_0(p)$. The mean value theorem implies
\[
 \na f(x)\cdot [(H_0(\na f_1(x))-H_0(\na f_2(x))]=\ell_x(\na f_1(x))-\ell_x(\na f_2(x))=\na f(x)\cdot \na \ell_x(p)
\]
for some $p=p_x$ on the line segment joining $\na f_1(x)$ and $\na f_2(x)$. Direct calculation gives
\[
\p_j \ell_x(p)=\frac{\p_jf(x)(1+|p|^2)-(p\cdot\na f(x))p_j}{(1+|p|^2)^\tdm},
\]
whence
\[
\na f(x)\cdot \na \ell_x(p)=\frac{|\na f(x)|^2(1+|p|^2)-(p\cdot\na f(x))^2}{(1+|p|^2)^\tdm}\ge \frac{|\na f(x)|^2}{(1+|p|^2)^\tdm}\ge 0.
\]
It follows that 
\bq\label{energy:H}
\frac{\sigma}{(1+(\max\{\| \na f_1\|_{L^\infty, \| \na f_2\|_{L^\infty}}\})^2)^\tdm}\int_{\T^d} |\na f(x)|^2dx+g\int_{\T^d} f^2dx\le \int_{\T^d}(h_1-h_2)fdx.
\eq
Therefore, we obtain the $H^1$ contraction estimate 
 \bq\label{H1contra:est}
\frac{\sigma}{M_0} \| \na (f_1-f_2)\|^2_{L^2}+g\| f_1-f_2\|^2_{L^2}\le g^{-1}\| h_1-h_2\|_{L^2}^2,
 \eq
  where $M_0=M_0(\| \na f_1\|_{L^\infty}, \| \na f_2\|_{L^\infty})$. This yields the desired uniqueness. 

Next, we prove the existence of $f\in C^{k, \a}(\T^d)$ solving \eqref{cgeq} for $h\in C^{k-2, \a}(\T^d)$, $k\ge 3$. We will only prove this for $k=3$ as the case $k\ge 4$ then follows  from  a standard  bootstrap.  We recall that  $-H$ is a  quasilinear elliptic operator:
\[
-H(u)=\frac{1}{(1+|\na u|^2)^\mez}\p_j^2u-\frac{\p_ju\na u}{(1+|\na u|^2)^\tdm}\cdot \na \p_ju\equiv a^{ij}(\na u)\p_{ij}u,
\]
where $a^{ij}\in C^\infty(\Rr^d)$ is given by
\[
a^{ij}(p)=\frac{\delta_{ij}}{(1+|p|^2)^\mez}-\frac{p_ip_j}{(1+|p|^2)^\tdm}
\]
and satisfies 
\bq
(1+|p|^2)^{-\tdm}|\xi|^2\le a^{ij}(p)\xi_i\xi_j\le (1+|p|^2)^{-\mez}|\xi|^2.
\eq
Hence, \eqref{cgeq} is equivalent to 
\bq\label{cgeq:1} 
\sigma a^{ij}(\na f(x)) \p^2_{ij} f(x)-g f(x)+h(x)=0.
\eq
By virtue of the Leray-Schauder  fixed point theorem (see Theorems 11.3 and 11.4 in \cite{GilTru}), it suffices to establish an a priori estimate for $\| f\|_{C^{2, \a}(\T^d)}$, assuming that $f\in C^{3, \a}(\T)$ is a solution.  It is readily seen that \eqref{cgeq} (or \eqref{cgeq:1}) obeys the maximum principle 
\bq\label{Cest}
\| f\|_{C(\T)}\le g^{-1}\| h\|_{C(\T)}. 
\eq  
To derive a bound for $\| \na f\|_{C(\T^d)}$ we first  rewrite \eqref{cgeq:1}  as
  \bq\label{cap-gravity rewritten}
       \sigma  \tilde{a}^{ij}(\nabla f) \partial_{ij} f = (gf - h) \sqrt{1+|\nabla f|^2},
    \eq
    where 
    \bq
        \tilde{a}^{ij}(\nabla f) = \delta^{ij} - \frac{\partial_i f \partial_j f}{1+ |\nabla f|^2}.
    \eq
We set $F(x) = \mez |\nabla f|^2$. Since $f\in C^3(\T^d)$, we have 
 \bq
    \partial_i F = \partial_k f \partial_{ki} f, \qquad \partial_{ij} F = \partial_{ki} f \partial_{kj}f + \partial_k f \partial_{kij} f,
 \eq
and hence
 \bq\label{aij F}
 \tilde{a}^{ij}(\nabla f) \partial_{ij}F  = |D^2 f|^2 - (1+|\nabla f|^2)^{-1}\p_if\partial_{ki} f \p_jf\partial_{kj} f +\partial_k f \tilde{a}^{ij}(\nabla f)\partial_k \p_{ij}f,
 \eq
 where we have used Einstein's summation convention.  To get rid of the third derivatives, we differentiate \eqref{cap-gravity rewritten} with respect to $x_k$:
\begin{multline*}
    \tilde{a}^{ij}(\nabla f)\partial_k \p_{ij}f  =   (1+|\nabla f|^2)^{-1}\bigl[ \partial_{j} f\partial_{ki} f  + \partial_i f\partial_{kj} f  \bigr] \partial_{ij} f 
    - 2 (1+|\nabla f|^2)^{-2} (\partial_\ell f \partial_{k\ell } f) (\partial_i f \partial_{j} f \partial_{ij}f)  \\
    + \sigma^{-1}(g\p_k f - \p_k  h )(1+|\nabla f|^2)^\mez + \sigma^{-1}(gf- h)(1+|\nabla f|^2)^{-\mez} \p_{k\ell}f \p_\ell f.
\end{multline*}
Plugging the above into \eqref{aij F}  gives
\begin{multline*}
    \tilde{a}^{ij}(\nabla f) \partial_{ij}F  = |D^2 f|^2 - (1+|\nabla f|^2)^{-1}\p_if\partial_{ki} f \p_jf\partial_{kj} f \\
    + (1+|\nabla f|^2)^{-1}\bigl[ \partial_{j} f(\p_kf\partial_{ki} f) + \partial_i f(\p_k f \partial_{kj} f)  \bigr] \partial_{ij} f 
    - 2 (1+|\nabla f|^2)^{-2} (\partial_\ell f \p_k f\partial_{k\ell } f) (\partial_i f \partial_{j} f \partial_{ij}f)  \\
    + \sigma^{-1}(g\p_k f\p_k f - \p_k  h \p_k f )(1+|\nabla f|^2)^\mez +  \sigma^{-1}(gf- h)(1+|\nabla f|^2)^{-\mez} (\p_k f\p_{k\ell}f) \p_\ell f.
\end{multline*}
There exists  $x_0\in \T^d$  such that $\max_{\T^d} F = F(x_0)$. Then, all the terms of the form $\sum_{\mu =1}^d\p_\mu  f \p_{\mu \nu} f$ in the above vanish at $x_0$, resulting in
$$\tilde{a}^{ij}(\nabla f) \partial_{ij}F(x_0) = |D^2f(x_0)|^2 +\sigma^{-1} (g|\nabla f(x_0)|^2 - \nabla f(x_0) \cdot \nabla h(x_0)) (1+|\nabla f(x_0)|^2)^\mez. $$
Ellipticity of $\tilde{a}^{ij}$ and the fact that $F$ attains its maximum at $x_0$ implies that $\tilde{a}^{ij}(\nabla f) \partial_{ij}F(x_0) \leq 0$, whence 
$$g|\nabla f(x_0)|^2 - \nabla f(x_0) \cdot \nabla h(x_0) \leq 0.$$
The  Cauchy-Schwarz inequality then implies 
\bq\label{C1est}
\| \na f\|_{C(\T^d)}\le g^{-1}\| \na h\|_{C(\T^d)}.
\eq
Combing \eqref{Cest} and \eqref{C1est} yields the $C^1$ estimate 
\bq\label{C1est:2}
\| f\|_{C^1(\T^d)}\le g^{-1} \| h\|_{C^1(\T^d)}.
\eq
By Theorem 13.1 in \cite{GilTru},  the divergence form equation \eqref{cgeq} obeys the $C^{1, \gamma}$ estimate 
\bq\label{C1aest}
\|  f\|_{C^{1, \gamma}(\T^d)}\le M_1(\| f\|_{C^1(\T^d)})\le M_1(g^{-1}\| h\|_{C^1(\T^d)})\quad\text{for~some~}\gamma=\gamma(\|h\|_{C^1})\in (0, 1).
\eq
This  yields the $C^{0, \gamma}(\T^d)$ estimate for the coefficients of the elliptic equation \eqref{cgeq:1}. Consequently, Schauder's estimates (see e.g. Theorem 6.2 in \cite{GilTru}) imply 
\bq\label{C2est}
\| f\|_{C^{2, \gamma}(\T^d)}\le M_2(\| h\|_{C^1(\T^d)}).
\eq 
By using \eqref{C2est} and reapplying Schauder's estimates, we obtain
\bq\label{C2est:1}
\| f\|_{C^{2, \beta}(\T^d)}\le M_2(\| h\|_{C^1(\T^d)})\quad\forall \beta \in (0, 1). 
\eq 
We have proven that the map $\sigma H+g I: C^{k, \a}(\T^d)\to C^{k-2, \a}(\T^d)$ is  bijective, $k\ge 3$.  To prove that the inverse map is continuous from $C^{k-2, \a}(\T^d)$ to $C^{k, \a}(\T^d)$, we note that  $f:=f_1-f_2$ obeys the equation 
\bq
Qf=a^{ij}(\na f_1)\p^2_{ij}f-\sigma^{-1}gf=-[a^{ij}(\na f_1)-a^{ij}(\na f_2)]\p^2_{ij}f_2-\sigma^{-1}(h_1-h_2):=k.
\eq
Since $f_j\in C^{k, \a}(\T^d)$, we have $a^{ij}(\na f_1)\in C^{k-1, \a}$ and  $k\in C^{k-2, \a}$ with
\bq\label{est:RHS:Qf}
\| k\|_{C^{k-2,\a}}\le M(\| f_1\|_{C^{k, \a}}, \|f_2\|_{C^{k, \a}})\left(\| f\|_{C^{k-1, \a}}+\| h_1-h_2\|_{C^{k-2, \a}}\right).
\eq
Invoking  Schauder's estimates again, we deduce that
\bq\label{schauder:f}
\| f\|_{C^{k, \a}}\le M\left(\| f\|_{C(\T^d)}+\| k\|_{C^{k-2, \a}}\right)\le M\left(\| f\|_{C^{k-1, \a}}+\| h_1-h_2\|_{C^{k-2, \a}}\right),
\eq 
 where $M=M(\| f_1\|_{C^{k, \a}}, \|f_2\|_{C^{k, \a}})$. By interpolating $\| f\|_{C^{k-1, \a}}$ between $\| f\|_{C^{k, \a}}$ and  the $H^1$ bound \eqref{H1contra:est}, we  deduce from \eqref{schauder:f} that
 \bq\label{contraest:proof:100}
 \| f_1-f_2\|_{C^{k, \a}}\le M(\| f_1\|_{C^{k, \a}}, \|f_2\|_{C^{k, \a}})\| h_1-h_2\|_{C^{k-2, \a}}.
\eq
Since $\| f_j\|_{C^{k, \a}}\le M(\| h_j\|_{C^{k-2, \a}})$ for $k\ge 3$,  the preceding estimate yields the Lipschitz continuity \eqref{continverse:H} of $(\sigma H+gI)^{-1}$. We note that \eqref{contraest:proof:100} holds for all $k\ge 2$.

Next, we let $h\in C^1(\T)$ and prove that \eqref{cgeq} has a unique solution $f\in C^{1, \beta}$ for all $\beta \in (0, 1)$. We first approximate $h$ by $h_n\in C^{1, \mez}(\T^d)$. The case $k=3$ above provides a unique function   $f_n\in C^{3, \mez}(\T^d)$  satisfying $(\sigma H+ gI)(f_n)=h_n$. Moreover, \eqref{C2est:1} implies the uniform bound
\[
\| f_n\|_{C^{2, \beta}(\T^d)}\le M_2(\| h_n\|_{C^1(\T^d)})\le M_2(2\| h\|_{C^1(\T^d)})\quad\forall \beta \in (0, 1).
\]
Then, applying the contraction estimate \eqref{contraest:proof:100} with $k=2$, we deduce that $\{f_n\}$ is a Cauchy sequence in $C^{2,  \beta}(\T^d)$ for all $\beta \in (0, 1)$. The limit $f\in \cap_{\beta \in (0, 1)} C^{2,  \beta}(\T^d)$ of $\{f_n\}$ is the unique solution of \eqref{cgeq}. 
\end{proof}
The next proposition provides the linearization of the mean curvature operator about zero. 
\begin{prop}\label{prop:linH}
Let $k\ge 2$ and $\a\in (0, 1)$. For any $\eta \in C^{k, \a}(\T^d)$ we have
\bq\label{lin:H}
H(\eta)=-\Delta \eta +R_H(\eta),
\eq
where $R_H(\eta)$ has mean zero and for some $M: \Rr_+\to \Rr_+$ depending only one $(d, k, \a)$, 
\bq\label{est:RH}
\| R_H(\eta)\|_{C^{k-2, \a}(\T^d)}\le M(\| \eta\|_{C^{k, \a}(\T^d)})\| \eta\|^2_{C^{k, \a}(\T^d)}.
\eq
Moreover, for any $\eta_1,~\eta_2\in  C^{k, \a}(\T^d)$ we have
\bq\label{contraction:RH}
\begin{aligned}
&\| R_H(\eta_1)-R_H(\eta_2)\|_{C^{k-2, \a}(\T^d)}\\
&\quad\le \tilde M\left(\| \eta_1\|_{C^{k, \a}(\T^d)}, \| \eta_2\|_{C^{k, \a}(\T^d)}\right)\left(\| \eta_1\|_{C^{k, \a}(\T^d)}+\| \eta_2\|_{C^{k, \a}(\T^d)}\right)\| \eta_1-\eta_2\|_{C^{k, \a}(\T^d)},
\end{aligned}
\eq
where $\tilde M: \Rr_+\times \Rr_+\to \Rr_+$ depends only on $(d, k, \a)$.
\end{prop}
\begin{proof}
This follows from the formula
\[
H(\eta)=-\Delta \eta-\di(\na \eta H_1(\na\eta)),
\]
where $H_1(p)=(1+|p|^2)^{-\mez}-1$ satisfies $H_1(0)=0$. 
\end{proof}
\section{Proof of Theorem \ref{theo:main}}\label{section:proof}
\subsection{Small traveling waves}
In this section we construct small traveling wave solutions  when the external pressure  is suitably small, proving in particular part I of Theorem \ref{theo:main}. For this purpose it is convenient to replace $\ka\varphi$ in \eqref{externalpressure} by a single function $\psi$, so that \eqref{eq:tw} becomes
\bq\label{eq:tw:0}
-\gamma \p_1\eta=-G[\eta](\sigma H(\eta)+g\eta+\psi).
\eq
 We will construct small solutions in $\rC^{3, \a}(\T^d)$ of \eqref{eq:tw:0} by the contraction mapping method. To proceed we first need to rewrite the traveling wave equation \eqref{eq:tw} in an appropriate fixed point formulation. Using the linearization of the Dirichlet-to-Neumann operator and the mean curvature operator   in Propositions \ref{prop:linDN} and \ref{prop:linH}, we obtain
\bq
G[\eta](\sigma H(\eta)+g\eta)=m(D)(-\sigma \Delta +g)\eta+ \sigma m(D)R_H(\eta)+R[\eta](\sigma H(\eta)+g\eta).
\eq
Then the traveling wave equation \eqref{eq:tw:0} can be rewritten as
\bq\label{eq:eta:1}
\left[-\gamma\p_1+m(D)(-\sigma \Delta +g)\right]\eta= -\sigma m(D)R_H(\eta)-R[\eta](\sigma H(\eta)+g\eta) -G[\eta]\psi.
\eq
By virtue of Propositions \ref{theo:iso:G} and \ref{prop:linDN}, the right-hand side of \eqref{eq:eta:1} has mean zero. Moreover, the symbol of the Fourier multiplier  $-\gamma\p_1+m(D)(-\sigma \Delta +g)$ vanishes only at $0$, hence we can invert it to obtain the  equivalent formulation
\bq
\eta=K^{\sigma, g}_\psi(\eta),\quad\int_{\T^d}\eta=0,
\eq
where
\bq
 K^{\sigma, g}_\psi(\eta)=[-\gamma\p_1+m(D)(-\sigma \Delta +g)]^{-1}\left\{-\sigma m(D)R_H(\eta)-R[\eta](\sigma H(\eta)+g\eta)- G[\eta]\psi\right\}.
\eq
The existence and uniqueness of small traveling waves is proven in the next proposition. 
\begin{prop}\label{prop:smallTW}
Consider $\sigma>0$ and $g>0$. Let  $\gamma \in \Rr$, $\a\in (0, 1)$, and $\psi\in \rC^{1, \a}(\T^d)$. There exist positive constants  $\delta_1$ and $C_1$ depending only on $(\sigma, g, \gamma, d, b, \a)$   such that if $\delta<\delta_1$ and $\|\psi\|_{\rC^{1, a}(\T^d)}<\delta/C_1$ then the following holds.

(i)  $K^{\sigma, g}_\psi$ has a unique fixed point  in the ball of radius $\delta$ centered at the origin in $\rC^{3, \a}(\T^d)$. 

(ii) The unique fixed point of $K^{\sigma, g}_\psi$  is a Lipschitz continuous function of $\psi$.
\end{prop}
\begin{proof}
We write 
\bq\label{Kpsi:L}
K^{\sigma, g}_\psi=m_1(D)L- m_1(D)G[\eta]\psi,
\eq
where
\bq\label{m12}
\begin{aligned}
&m_1(\xi)=[-\gamma i\xi_1+m(\xi)(\sigma |\xi|^2+g)]^{-1},\\
&L(\eta)=-\sigma m(D)R_H(\eta)-R[\eta](\sigma H(\eta)+g\eta).
\end{aligned}
\eq
Since $m(\xi)\asymp |\xi|$ as $|\xi|\asymp \infty$, the symbol $m_1$ satisfies \eqref{multiplier} with $\tau=-3$ due to the presence of $\sigma>0$. Then  Propositions \ref{prop:multiplier} and \ref{prop:ZH} imply that $m_1(D): \rC^{k-3, \a}\to \rC^{k, \a}$ for $k\ge 3$. We also note that $m$ satisfies \eqref{multiplier} with $\tau=1$. Let $B_\delta$ denote the ball in $\rC^{3, \a}(\T^d)$ with center $0$ and radius $\delta$. 

We first take $\eta\in B_\delta$ with $\delta<b/2$ so that $\dist(\eta, -b)\ge b/2>0$.  Then using \eqref{est:RDN} and \eqref{est:RH}  we deduce that $L(\eta)$ belongs to  $\rC^{0, \a}$ and is at least quadratic in $\eta$, that is,
\bq
\| L(\eta)\|_{\rC^{0, \a}}\le C\| \eta\|^2_{\rC^{3, \a}}.
\eq
 It follows that 
\bq\label{bound:K}
\| K^{\sigma, g}_\psi(\eta)\|_{\rC^{3, \a}}\le C(\| \eta\|^2_{\rC^{3, \a}}+\| \psi\|_{\rC^{1, \a}}),
 \eq
 where the constant $C$ depends only on $(\sigma, g, \gamma, d, b,  \a)$.  Thus, if $\delta<\delta_1=\min\{1, (2C)^{-1}\}$ then for  $\| \psi\|_{\rC^{1, \a}}<  \delta/(2C)$ we have $K^{\sigma, g}_\varphi: B_\delta\to B_\delta$. 
 
 Next, using the boundedness \eqref{est:RDN} and the contraction estimates \eqref{contraction:RDN} and \eqref{contraction:RH}, we find
 \bq
 \begin{aligned}
 \| L(\eta_1)-L(\eta_2)\|_{\rC^{0, \a}}&
 \le C \| R_H(\eta_1)-R_H(\eta_2)\|_{\rC^{1, \a}}\\
 &\qquad+\| R[\eta_1](\sigma H(\eta_1)+g\eta_1)-R[\eta_2](\sigma H(\eta_1)+g\eta_1)\|_{\rC^{0, \a}}\\
 &\qquad+\| R[\eta_2]\left\{\sigma [H(\eta_1)-H(\eta_2)]+g(\eta_1-g\eta_2)\right\}\|_{\rC^{0, \a}}\\
 &\le  C(\| \eta_1\|_{\rC^{3, \a}}+\| \eta_2\|_{\rC^{3, \a}})\| \eta_1-\eta_2\|_{\rC^{3, \a}}
 \end{aligned}
 \eq
 and 
 \bq
 \| G[\eta_1]\psi- G[\eta_2]\psi\}\|_{\rC^{0, \a}} \le C\| \eta_1-\eta_2\|_{\rC^{1, \a}}\| \psi\|_{\rC^{1, \a}}.
 \eq
 It follows that 
 \bq\label{Lip:K}
 \begin{aligned}
 \| K^{\sigma, g}_\psi(\eta_1)- K^{\sigma, g}_\psi(\eta_2)\|_{\rC^{3, \a}}&= \| m_1(D)[L(\eta_1)- L(\eta_2)]\|_{\rC^{3, \a}}+\| m_1(D)\{ G[\eta_1]\psi- G[\eta_2]\psi\}\|_{\rC^{3, \a}} \\
 &\le C\| L(\eta_1)-L(\eta_2)\|_{\rC^{0, \a}}+C\| G[\eta_1]\psi- G[\eta_2]\psi\}\|_{\rC^{0, \a}}\\
 &\le C\big(\| \eta_1\|_{\rC^{3, \a}}+\| \eta_2\|_{\rC^{3, \a}}+\| \psi\|_{\rC^{1, \a}}\big)\| \eta_1-\eta_2\|_{\rC^{3, \a}}.
 \end{aligned}
 \eq
By shrinking $\delta$ if necessary we obtain that $K^{\sigma, g}_\psi: B_\delta\to B_\delta$ is a contraction. The Banach contraction mapping theorem then implies that $K^{\sigma, g}_\psi$ has a unique fixed point $\eta$ in $B_\delta$ provided $\| \psi\|_{\rC^{1, \a}}<\delta/(2C)$. Finally, the Lipschitz continuous dependence of $\eta$ on $\psi$ follows from \eqref{Lip:K} and the  estimate 
 \begin{align*}
  \| K^{\sigma, g}_{\psi_1}(\eta)- K^{\sigma, g}_{\psi_2}(\eta)\|_{\rC^{3, \a}}&=\| m_1(D)G[\eta](\psi_1-\psi_2)\|_{\rC^{3, \a}}\\
  &\le M(\| \eta\|_{\rC^{3, \a}})\| \psi_1-\psi_2\|_{\rC^{1, \a}}.
 \end{align*}
\end{proof}
Part I  of Theorem \ref{theo:main} follows from Proposition \ref{prop:smallTW} by setting $\psi=\ka \varphi$. When $\sigma =0$, by straightforward modifications of the proof of Proposition \ref{prop:smallTW},  we obtain a local curve of small traveling gravity waves in $\rC^{k, \a}(\T^d)$ for $k\ge 1$. We record this in the following proposition, which is a sharper version  of \cite[Theorem 6.3]{NguyenTice} stated for $\rH^s(\T^d)$ solutions with $s>1+\frac{d}{2}$.
\begin{prop}\label{prop:smallTW:g}
Consider $\sigma=0$ and $g>0$. Let  $\gamma \in \Rr$, $\a\in (0, 1)$, and $\psi\in \rC^{3, \a}(\T^d)$. There exist positive constants $\delta_0$ and $C_0$ depending only on $(g, \gamma, d, b,  \a)$   such that if $\delta<\delta_0$ and $\|\psi\|_{\rC^{3, a}(\T^d)}<\delta/C_0$ then the following holds.

(i)  $K^{0, g}_\psi$ has a unique fixed point in the ball of radius $\delta$ centered at the origin in $\rC^{3, \a}(\T^d)$. 

(ii) The unique fixed point of $K^{0, g}_\psi$  is a Lipschitz continuous function of $\psi$.

\end{prop}
Next we establish the connection between traveling capillary-gravity waves and traveling gravity waves through the vanishing surface tension limit.
\begin{prop}\label{prop:stlim}
Consider $\sigma\in (0, 1)$ and $g> 0$. Let  $\gamma \in \Rr$, $\a\in (0, 1)$, and $\psi\in \rC^{3, \a}(\T^d)$. The following assertions hold. 

(i)  There exist positive constants $\delta_u$ and $C_u$ independent of $\sigma \in (0, 1)$ and depending only on $(g, \gamma, d, b,  \a)$ such that for  any $\delta<\delta_u$ and $\| \psi\|_{\rC^{3, \a}}<\delta/C_u$ there exists a unique traveling capillary-gravity wave in $\rC^{3, \a}(\T^d)$ with norm less than $\delta$.  

(ii) Let $\a'\in (0, \a)$, $\delta<\min\{\delta_u(\a), \delta_u(\a'), \delta_0\}$,  and assume that 
\bq\label{cd:psi:lim}
\| \psi\|_{\rC^{3, \a}}< \delta/\max\{C_u(\a), C_u(\a'), C_0\}.
\eq
Let  $\sigma=\sigma_n\to 0$ and  let $\eta_n\in \rC^{3, \a}(\T^d)$ be the  above unique traveling capillary-gravity waves with $\sigma=\sigma_n$.  Then $\{\eta_n\}$ converges to the unique traveling gravity wave in $\rC^{3, \a'}(\T^d)$. \end{prop}
\begin{proof}
We recall \eqref{m12} for the symbol $m_1$ and set $m_\sigma=\sigma m_1$.  It can be readily checked that $m_1$  satisfies \eqref{multiplier} with $m=-1$ uniformly in $\sigma \in (0, 1)$. On the other hand, thanks to the factor $\sigma$ in $m_\sigma$, $m_\sigma$ satisfies \eqref{multiplier} with $m=-3$ uniformly in $\sigma \in (0, 1)$. Now we have
\bq
K^{\sigma, g}_\psi(\eta)=-m_\sigma(D)\left\{m(D)R_H(\eta)+R[\eta]H(\eta)\right\}-gm_1(D)R[\eta]\eta-m_1(D)G[\eta]\psi.
\eq
Therefore there exists $C_1>0$ independent of $\sigma \in (0, 1)$ such that 
\begin{align*}
\| K^{\sigma, g}_\psi(\eta)\|_{\rC^{3, \a}}&\le C_1\| m(D)R_H(\eta)+R[\eta]H(\eta)\|_{\rC^{0, \a}}\\
&\qquad+C_1\| R[\eta]\eta\|_{\rC^{2, \a}}+C_1\| G[\eta]\psi\|_{\rC^{2, \a}}.
\end{align*}
We consider $\| \eta\|_{\rC^{3, \a}}<b/2$ so that $\dist(\eta, -b)\ge b/2>0$. Then the remainder estimates \eqref{est:RDN} and \eqref{est:RH} yield 
\bq\label{uniest:K}
\| K^{\sigma, g}_\psi(\eta)\|_{\rC^{3, \a}}\le C_2\| \eta\|_{\rC^{3, \a}}^2+C_2\| \psi\|_{\rC^{3, \a}},
\eq
where $C_2>0$ is independent of $\sigma \in (0, 1)$ and depends only on $(g, \gamma, d, b, \a)$. By an analogous argument using the contraction estimates \eqref{contraction:RDN} and \eqref{contraction:RH}, we find 
\bq\label{uniestcontra:K}
\| K^{\sigma, g}_\psi(\eta_1)-K^{\sigma, g}_\psi(\eta_2)\|_{\rC^{3, \a}}\le C_3(\| \eta_1\|_{\rC^{3, \a}}+\| \eta_2\|_{\rC^{3, \a}})\| \eta_1-\eta_2\|_{\rC^{3, \a}},
\eq
where $C_3>0$ is independent of $\sigma \in (0, 1)$  and depends only on $(g, \gamma, d, b,  \a)$. In view of  \eqref{uniest:K} and \eqref{uniestcontra:K}, there exist positive constants $\delta_u$ and $C_u$ independent of $\sigma \in (0, 1)$ such that for  any $\delta<\delta_u$ and $\| \psi\|_{\rC^{3, \a}}<\delta_u/C_u$, $K^{\sigma, g}_\psi$ has a unique fixed point in the ball $B_{\delta}$ of $\rC^{3, \a}(\T^d)$. Note that $\delta_u$ and $C_u$ depend only on $(g, \gamma, d, b,  \a)$.

Fix $\a'\in (0, \a)$. Let $\delta<\min\{\delta_u(\a), \delta_u(\a'), \delta_0\}$  and assume that $\psi$ satisfies \eqref{cd:psi:lim}. Let $\sigma_n\to 0$ and let $\eta_n$ be the unique fixed point of $K^{\sigma_n, g}_\psi$ in $\rC^{3, \a}(\T^d)$ with norm less than $\delta$.   Since $\T^d$ is compact, combining the Arzel\`a-Ascoli theorem  with interpolation, we obtain a subsequence $\{\eta_{n_k}\}$ converging strongly to some $\eta_0$ in $C^{3, \a'}(\T^d)$.  Now we recall that the equation $\eta=K^{\sigma, g}(\eta)$ is equivalent to \eqref{eq:tw:0}, hence 
\[
 -\gamma \p_1\eta_n=-G[\eta_n]\left(\sigma_n H(\eta_n)+g\eta_n+\psi\right).
\]
 By virtue of the contraction estimates \eqref{contraction:RDN} and \eqref{contraction:RH}, one can pass to the limit in the topology of $C^{0, \a'}(\T^d)$ along the subsequence $n_k$ to deduce that $\eta_0\in \rC^{3, \a'}$ is a traveling gravity   wave.  In addition, we have $\| \eta_0\|_{\rC^{3, \a'}}\le \delta <\delta_0$ and $\| \psi\|_{\rC^{k, \a'}}\le \| \psi\|_{\rC^{k, \a}}<\delta_0/C_0$, hence  $\eta_0$ is the unique traveling gravity wave according to Proposition \ref{prop:smallTW:g}. Therefore, the entire sequence $\{\eta_n\}$ converges to $\eta_0$.  The proof of the proposition is complete. 

\end{proof}
 The next section is devoted to the proof of parts II and III of Theorem \ref{theo:main}.
\subsection{Large traveling waves} 
By virtue of Propositions \ref{theo:iso:G} and \ref{theo:iso:H} we can equivalently rewrite \eqref{eq:tw}  as
\bq\label{def:cF}
\cF(\eta, \ka):=\eta+F(\eta, \ka)=0
\eq
where 
\bq\label{def:F}
F(\eta, \ka)=  (\sigma H+g I)^{-1}\left(-\gamma(G[\eta])^{-1} \p_1\eta+\ka \varphi\right).
\eq
The parameters are 
\bq
\sigma>0,\quad g> 0,\quad \gamma \in \Rr\setminus\{0\}.
\eq
For $k\ge 0$ and $\a\in (0, 1)$ we denote 
\bq
V_{k, \a}=
\begin{cases}\left \{\eta\in \rC^{k, \a}(\T^d): \inf_{x\in \T^d}(\eta(x)+b)>0\right\}\quad\text{for finite depth},\\
\rC^{k, \a}(\T^d)\quad\text{for infinite depth}.
\end{cases}
\eq
We begin by proving the regularity statement III in Theorem \ref{theo:main}.
\begin{lemm}
Suppose that $\eta\in \rC^{3, \a}(\T^d)$ satisfies \eqref{def:cF} for some $\ka \in \Rr$ and $\varphi\in \rC^{1, \a}(\T^d)$. For any $k\ge 3$ and $\mu\in (0, 1)$, if $\varphi\in \rC^{k-2, \mu}(\T^d)$ then $ \eta \in \rC^{k, \mu}(\T^d)$. 
\end{lemm}
\begin{proof}
Combining Propositions \ref{theo:iso:G} and \ref{theo:iso:H}, we obtain
\[
(\sigma H+g I)^{-1}: \rC^{j, \nu}(\T^d)\to \rC^{j+2, \nu}(\T^d)\]

and 
\[
\rC^{j, \nu}(\T^d)\ni \eta\mapsto (G[\eta])^{-1}\p_1\eta \in \rC^{j, \mu}(\T^d)
\]
for $j\ge 1$. 
We first consider the case $k=3$, $\mu\in (0, 1)$, and  assume $\varphi\in \rC^{1, \mu}$. Then we have $-\gamma(G[\eta])^{-1} \p_1\eta+\ka \varphi\in \rC^{3, \alpha}+\rC^{1, \mu}\subset \rC^{1, \mu}$, and hence $\eta=-F(\eta, \ka)\in  \rC^{3, \mu}$.

Next, we consider $k\ge 4$, $\mu\in (0, 1)$, and assume  $\varphi\in \rC^{k-2, \mu}$. We have $\eta \in \rC^{3, \mu}$ from the case $k=3$.  Consequently $-\gamma(G[\eta])^{-1} \p_1\eta+\ka \varphi\in \rC^{3, \mu}+\rC^{k-2, \mu}= \rC^{\min\{k-2, 3\}, \mu}$ and $\eta=-F(\eta, \ka)\in \rC^{\min\{k-2, 3\}, \mu}= \rC^{\min\{k, 5\}, \mu}$. But this in turn implies $\eta=-F(\eta, \ka)\in \rC^{\min\{k, 5, k-2\}+2, \mu}=\rC^{\min\{k, 7\}, \mu}$ by the same argument. After finitely many steps we obtain that $ \eta \in \rC^{k, \mu}$.
\end{proof}
The remainder of this section is devoted to the proof of part II of Theorem \ref{theo:main}.  We first establish the compactness of the nonlinear operator $F$.
\begin{lemm}
For $k\ge 3$, $\a\in (0, 1)$ and $\varphi\in \rC^{k-2, \a}(\T^d)$, $F: V_{k, \a}\times \Rr\to \rC^{k, \a}(\T^d)$ is a compact operator. Moreover, $F$ is Lipschitz continuous on  bounded subsets  of  $V_{k, \a}\times \Rr$.
\end{lemm}
\begin{proof}
Let $\{(\eta_n, \kappa_n)\}$ be a bounded sequence in  $ V_{k, \a}\times \Rr$. By virtue of the Arzel\`a-Ascoli theorem and interpolation,  the embedding $\rC^{k, \a}(\T^d)\to\rC^{k-1, \a}(\T^d)$ is compact, so upon extracting a subsequence, we have $(\eta_n, \kappa_n)\to (\eta, \ka)$ in $V_{k-1, \a}\times \Rr$. We claim that the operator $F_0$ defined by 
\[
 F_0(\eta):=(G[\eta])^{-1} \p_1\eta,
\]
is continuous from $V_{k-1, \a}$ to  $\rC^{k-1, \a}(\T^d)$. Taking this for granted and using the assumption that $\varphi\in \rC^{k-2, \a}(\T^d)$, we deduce  
 \[
-\gamma(G[\eta_n])^{-1} \p_1\eta_n+\ka_n \varphi\to -\gamma(G[\eta])^{-1} \p_1\eta+\ka \varphi\quad\text{in~} \rC^{k-2, \a}(\T^d).
\]
Then, by virtue of Proposition \ref{theo:iso:H}, we have $F(\eta_n, \kappa_n)\to F(\eta, \kappa)$ in $\rC^{k, \a}(\T^d)$, proving the compactness of $F$.  

Since $k-1> 1$, in order to obtain the above claim, we now prove that $F_0$ is continuous from $V_{j, \a}$ to $\rC^{j, \a}(\T^d)$ for any $j\ge 1$.  For any $\eta_0\in V_{j, \a}$, Proposition \ref{theo:iso:G} implies that $\eta\mapsto (G[\eta_0])^{-1}\p_1\eta: V_{j, \a}\to \rC^{j, \a}(\T^d)$ is  a bounded linear  operator  with norm bounded by $C(\| \eta_0\|_{\rC^{j, \a}})$. On the other hand, for any $f\in\rC^{j-1, \a}(\T^d)$ and $\eta_j\in V_{j, \a}$, we have
\bq
(G[\eta_1])^{-1}f-(G[\eta_2])^{-1}f=(G[\eta_1])^{-1}\left(G[\eta_2]h-G[\eta_1]h\right),\quad h:=(G[\eta_2])^{-1}f.
\eq
Applying \eqref{inverseDN} for the boundedness of $(G[\eta_j])^{-1}$ and \eqref{contraction:RDN} for the contraction $G[\eta_1]-G[\eta_2]=R[\eta_1]-R[\eta_2]$, we find  
\begin{align*}
\| (G[\eta_1])^{-1}f-(G[\eta_2])^{-1}f\|_{\rC^{j, \a}}&\le M_0(\| \eta_1\|_{\rC^{j, \a}})\| G[\eta_2]h-G[\eta_1]h\|_{\rC^{j-1, \a}}\\
&\le M_1(\| \eta_1\|_{\rC^{j, \a}}, \| \eta_2\|_{\rC^{j, \a}})\|\eta_1-\eta_2\|_{\rC^{j-1, \a}}\| h\|_{\rC^{j, \a}}\\
& \le M_2(\| \eta_1\|_{\rC^{j, \a}}, \| \eta_2\|_{\rC^{j, \a}})\|\eta_1-\eta_2\|_{\rC^{j, \a}}\| f\|_{\rC^{j-1, \a}}. 
\end{align*}
This actually shows that for any $f\in \rC^{j-1, \a}(\T^d)$ the operator $V_{j, \a}\ni \eta\mapsto (G[\eta])^{-1}f\in \rC^{j, \a}(\T^d)$ is Lipschitz continuous on bounded subsets of $V_{j, \a}$, $j\ge 1$. Consequently $F_0$ is Lipschitz continuous on bounded subsets of $V_{j, \a}$, $j\ge 1$. By Proposition \ref{theo:iso:H}, $(\sigma H+g I)^{-1}: \rC^{j, \a}(\T^d)\to \rC^{j+2, \a}(\T^d)$ is Lipschitz continuous on bounded subsets of $\rC^{j, \a}(\T^d)$, $j\ge 1$. Therefore, for $k\ge 3$ and $\varphi\in \rC^{k-2, \a}(\T^d)$, $F$ is Lipschitz continuous on  bounded subsets of  $V_{k, \a}\times \Rr$.
\end{proof}
The compactness of $F$ motivates  us to invoke the following Global Implicit Function Theorem which is based upon the Leray-Schauder degree theory.
\begin{theo}(Global Implicit Function Theorem)\label{GIFT}
Let $X$  be  Banach space and let $U$ be an open subset of $X\times \Rr$. Let $F: U\to X$ be a compact operator and set $\cF(\eta, \ka)=\eta+F(\eta, \ka)$.  Assume that $\cF(\eta_0, \ka_0)=0$ for some $(\eta_0, \ka_0)\in X\times \Rr$, and the Fr\'etchet derivative  $D_\eta\cF(\eta_0, \ka_0)\in L(X, X)$ exists and is bijective. Let $\mathcal{S}$ denote the closure in $X\times \Rr$ of the set of all solutions of $\cF(\eta, \ka)=0$ in $U$ and let $\mathcal{C}$ denote the connected component of $\mathcal{S}$ to which $(\eta_0, \ka_0)$ belongs. Then one of the following three alternatives is valid.
\begin{itemize}
\item[(i)] $\mathcal{C}$ is unbounded.
\item[(ii)] $ \mathcal{C}\setminus\{(\eta_0, \ka_0)\}$ is connected.
\item[(iii)] $\mathcal{C}\cap \p U\ne \emptyset$. 
\end{itemize}
\end{theo}
\begin{proof}
For $U=X\times \Rr$ and under the additional assumption that the Frech\'et derivative $D_\eta F$ is continuous near $(\eta_0, \ka_0)$, this  is Theorem II.6.1 in \cite{Kielhofer}. As pointed out  in Remark II.6.2 in \cite{Kielhofer}, the preceding assumption can be dropped at the price that the local curve of solutions near $(\eta_0, \ka_0)$ is not necessarily unique. On the other hand, the proofs of the cited results extend trivially to the case of general open set $U$ in $X\times \Rr$ at the price of adding the third alternative (iii). 
\end{proof}
We  apply Theorem \ref{GIFT} to our nonlinear operator $F$ \eqref{def:F} with 
\begin{align}\label{defX}
 &X=\rC^{3, \a}(\T^d),~\varphi\in \rC^{1, \a}(\T^d)\setminus\{0\},~ \a\in (0, 1),\\
 &U=V\times \Rr,\quad V:= V_{3, \a}.
\end{align}
We also define the projections
\bq\label{def:Pj}
\Pi_1(u, \ka)=u,\quad \Pi_2(u, \ka)=\ka.
\eq
 Clearly $\cF(0, 0)=0$. We note that although the uniqueness of the local curve of solutions near $(\eta, \ka)=(0, 0)$ does not follow from Theorem \ref{GIFT}, it is guaranteed by Proposition \ref{prop:smallTW}. Next, we prove that $\eta=0$ is the unique solution in $V$ of $\cF(\cdot, 0)=0$. This will in turn allow us to exclude the `loop' scenario in the alternative (ii).
\begin{lemm}\label{lemm:rigidity}
 $\eta=0$ is the unique solution in $V$ of the equation $\eta+F(\eta, 0)=0$.
\end{lemm}
\begin{proof}
Suppose that $\eta\in V$ satisfies $\eta+F(\eta, 0)=0$ which is equivalent to \eqref{eq:tw} with $\ka=0$, i.e.
\bq\label{eq:tw:free}
-\gamma \p_1\eta=-G[\eta](\sigma H(\eta)+g\eta).
\eq
Then Theorem \ref{theo:iso:G} implies that $\sigma H(\eta)+g\eta=\gamma (G[\eta])^{-1}\p_1\eta\in \rC^{3, \a}(\T^d)$.  Multiplying \eqref{eq:tw:free} by $\sigma H(\eta)+g\eta$ and integrating, we find
\bq
\begin{aligned}
-\int_{\T^d}(\sigma H(\eta)+g\eta)G[\eta](\sigma H(\eta)+g\eta)&= -\gamma \int_{\T^d}\p_1\eta (\sigma H(\eta)+g\eta)dx\\
&=-\mez\gamma \sigma  \int_{\T^d}\frac{\p_1|\na \eta|^2}{\sqrt{1+|\na \eta|^2}}dx-\mez \gamma g\int_{\T^d} \p_1 \eta^2dx\\
&=-\gamma \sigma \int_{\T^d}\p_1\sqrt{1+|\na \eta|^2}dx=0.
\end{aligned}
\eq
 Hence $\sigma H(\eta)+g\eta=0$ since $\sigma H(\eta)+g\eta\in  \rC^{3, \a}(\T^d)\subset \rH^\mez(\T)$ and $G[\eta]$ is positive-definite on $\rH^\mez(\T)$ by Proposition \ref{DN:positive}. Therefore $\eta=0$ by Proposition \ref{theo:iso:H}. Alternatively, we can integrate by parts \[
0=g\int_{\T^d}\eta^2dx-\sigma\int_{\T^d}\eta  \di \left(\frac{\na \eta}{\sqrt{1+|\na \eta|^2}}\right)dx=g\int_{\T^d}\eta^2dx+\sigma\int_{\T^d}  \frac{|\na \eta|^2}{\sqrt{1+|\na \eta|^2}}dx
\]
to deduce $\eta=0$.
\end{proof}
We verify the conditions for $\cF$  in the next lemma.
\begin{lemm}
$\cF(\cdot, 0)$ is Fr\'etchet differentiable at $0$ and 
\bq\label{formula:DF}
D_\eta\cF(0, 0)=I-\gamma(-\sigma \Delta+gI)^{-1}m^{-1}(D)\p_1,
\eq
 where $m$ is given by \eqref{def:m}. Moreover, $D_\eta\cF(0, 0):X\to X$ is an isomorphism. 
\end{lemm}
\begin{proof}
Assuming for the moment that  \eqref{formula:DF} holds, then $D_\eta\cF(0, 0)$ is the Fourier multiplier 
\[
a(\xi)=1-i\gamma(\sigma |\xi|^2+g)^{-1}m^{-1}(\xi)\xi_1,
\]
which only vanishes at $\xi=0$. It can be readily checked that  $a$ and $a^{-1}$  satisfy \eqref{multiplier} with $\tau=0$. Then Propositions \ref{prop:multiplier} and \ref{prop:ZH} imply that $D_\eta\cF(0, 0): \rC^{3, \a}(\T^d)\to \rC^{3, \a}(\T^d)$ is an isomorphism. Therefore, it remains to prove \eqref{formula:DF}. Denoting $T(\eta)= (G[\eta])^{-1}(\gamma \p_1\eta): \rC^{3, \a}(\T^d)\to \rC^{3, \a}(\T^d)$, the continuity estimate \eqref{inverseDN} for $(G[\eta])^{-1}$ implies
\bq\label{est:opT}
\| T (f)\|_{\rC^{3, \a}}\le M(\| f\|_{\rC^{3, \a}})\| f\|_{\rC^{3, \a}}.
\eq
The  expansion $G[f]=m(D)+R[f]$ gives
\[
\gamma\p_1f=G[f]T(f)=m(D)T(f)+R[f]T(f).
\]
Since $R[f]$ and $\p_1f$ have mean zero, we can invert $m(D)$ and obtain
\bq\label{form:Tf}
T(f)=\gamma m^{-1}(D)\p_1f-m^{-1}(D)R[f]T(f),
\eq
where $m^{-1}(D)\p_1f$ is linear in $f$. Using \eqref{est:opT} and the remainder estimate \eqref{est:RDN}, we find 
\bq\label{lin:T}
\begin{aligned}
\| m^{-1}(D)R[f]T(f)\|_{\rC^{3, \a}}&\le C\| R[f]T(f)\|_{\rC^{2, \a}}\\
&\le  M(\| f\|_{\rC^{3, \a}})\| f\|_{\rC^{3, \a}}\| T(f)\|_{\rC^{3, \a}} \le M(\| f\|_{\rC^{3, \a}})\| f\|_{\rC^{3, \a}}^2.
\end{aligned}
\eq
Since $T(0)=0$, this proves that  $T$ is Fr\'etchet differentiable at $0$ and $DT(0)=\gamma m^{-1}(D)\p_1$. 

Next we denote $S=(\sigma H+gI)^{-1}$, so that $\cF(\cdot, 0)=ST$.  The linearization \eqref{lin:H} implies
\[
T(f)=(\sigma H+gI)ST(f)=(-\sigma \Delta+gI)ST(f)+R_HST(f).
\]
Inserting \eqref{form:Tf} for $T(f)$ yields
\bq\label{lin:ST}
\begin{aligned}
ST(f)&=(-\sigma \Delta+gI)^{-1}T(f)-(-\sigma \Delta+gI)^{-1}R_HST(f)\\
&=\gamma(-\sigma \Delta+gI)^{-1}m^{-1}(D)\p_1f\\
&\qquad-(-\sigma \Delta+gI)^{-1}\left\{R_HST(f)+m^{-1}(D)R[f]T(f)\right\}.
\end{aligned}
\eq
Since both $R_H$ and $R[f]$ have mean zero, the remainder estimate \eqref{est:RH} implies
\bq\label{lin:ST:1}
\begin{aligned}
&\| (-\sigma \Delta+gI)^{-1}\left\{R_HST(f)+m^{-1}(D)R[f]T(f)\right\}\|_{\rC^{5, \a}}\\
&\le C\| R_HST(f)+m^{-1}(D)R[f]T(f)\|_{\rC^{3, \a}}\\
&\le M(\|ST(f)\|_{C^{5, \a}})\|ST(f)\|_{C^{5, \a}}^2+\|m^{-1}(D)R[f]T(f)\|_{\rC^{3, \a}}.
\end{aligned}
\eq
We deduce from \eqref{est:opT} and the Lipschitz continuity  \eqref{continverse:H} for $S$ that
\[
\|ST(f)\|_{C^{5, \a}}\le M(\| T(f)\|_{C^{3, \a}})\| T(f)\|_{C^{3, \a}}\le M(\|f\|_{\rC^{3, \a}})\| f\|_{\rC^{3, \a}}.
\]
Combining this with \eqref{lin:T} and \eqref{lin:ST:1} yields 
\bq\label{lin:ST:2}
\begin{aligned}
\| (-\sigma \Delta+gI)^{-1}\left\{R_HST(f)+m^{-1}(D)R[f]T(f)\right\}\|_{\rC^{5, \a}}\le M(\|f\|_{\rC^{3, \a}})\| f\|_{\rC^{3, \a}}^2.
\end{aligned}
\eq
In view of \eqref{lin:ST} and \eqref{lin:ST:2} we conclude that  $D(ST)(0)=\gamma(-\sigma \Delta+gI)^{-1}m^{-1}(D)\p_1$ which in turn yields \eqref{formula:DF}.
\end{proof}
We shall first prove part II of Theorem \ref{theo:main} for the finite depth case, then describe necessary modifications for the infinite depth case. 
\subsubsection{Finite depth}
Applying Theorem \ref{GIFT}, 
 one of the three alternatives (i), (ii) and (iii) is valid. If (iii) holds, i.e. $ \emptyset \ne\mathcal{C}\cap \p U=\mathcal{C}\cap (\p V\times \Rr)$, then there is a sequence of traveling solutions $\eta_n$ touching the bottom as $n\to \infty$. 
 
 Assume from now on that (iii) is false, so $\mathcal{C}\subset V\times \Rr$ and there exists $c^0>0$ such that 
 \bq\label{ulowerb}
 \inf_{x\in \T^d}(\eta(x)+b)\ge c^0 \quad\forall (\eta, \ka)\in  \mathcal{C}.
 \eq
\eqref{ulowerb} allows us to apply the results for the Dirichlet-to-Neumann operator established in Section \ref{section:DN}. 

  Assume also for the sake of contradiction that (ii) holds, i.e. $ \mathcal{C}\setminus\{(0, 0)\}$ is connected. By virtue of Lemma \ref{lemm:rigidity}, $V\times \{0\}$ and $ \mathcal{C}\setminus\{(0, 0)\}$ are disjoint, so
 \[
 \mathcal{C}\setminus\{(0, 0)\}\subset V\times (\Rr\setminus\{0\})=(V\times (-\infty, 0))\cup (V\times (0, \infty)).
\]
  Since each set on the right-hand side is open in $X\times \Rr$, $ \mathcal{C}\setminus\{(0, 0)\}$ is disconnect, a contradiction. Thus (ii) is false and (i) must hold, i.e. $\mathcal{C}$ is unbounded. There are two possibilities. 

{\bf Case 1:} $\ka$ is bounded in $\mathcal{C}$, i.e. $\sup_{\ka \in \Pi_2\mathcal{C}}|\ka|<\infty$. 
 Then $\eta$ must be unbounded in $\mathcal{C}$, that is, 
\bq\label{unb:0}
\sup_{\eta \in \Pi_1\mathcal{C}} \| \eta\|_{C^{3, \a}(\T^d)}=\infty.
\eq
 In fact, we have
 \begin{lemm}
 For any $\beta \in (0, \a)$, it holds that
  \bq\label{unb:1}
\sup_{\eta \in \Pi_1\mathcal{C}} \| \eta\|_{C^{1, \beta}(\T^d)}=\infty.
\eq
\end{lemm}
\begin{proof}
By virtue of \eqref{ulowerb}, the estimate \eqref{inverseDN} holds.  Combining \eqref{inverseDN} and \eqref{continverse:H} yields
\bq\label{estF:Ck}
\begin{aligned}
\| F(\eta, \ka)\|_{C^{3, \a}}&\le  M_1\left(\| \gamma (G[\eta])^{-1}\p_1\eta \|_{C^{1, \a}}+\| \ka \varphi \|_{C^{1, \a}}\right)\left(\| \gamma G[\eta])^{-1}\p_1\eta\|_{C^{1, \a}}+\|\ka \varphi\|_{C^{1, \a}}\right)\\
&\le M_2\left(\|\eta \|_{C^{1, \a}}+|\ka|\|\varphi \|_{C^{1, \a}}\right)\left(\| \eta\|_{C^{1, \a}}+ |\ka|\| \varphi\|_{C^{1, \a}}\right),
\end{aligned}
\eq
where $M_1$ depends only on $(\sigma, g, d, c^0, \beta)$, and $M_2$ depends only on $(\sigma, g, d, c^0, \beta, |\gamma|)$. Since $\eta=-F(\eta, \ka)$ in $\mathcal{C}$, it  follows from \eqref{unb:0},  \eqref{estF:Ck}, and the boundedness of $\ka$ in $\mathcal{C}$ that 
\bq\label{unb:C1a}
\sup_{\eta \in \Pi_1\mathcal{C}} \| \eta\|_{C^{1, \a}(\T^d)}=\infty.
\eq
Now let  $\beta \in (0, \a)$.  We apply \eqref{estF:Ck} with $\a$ replaced by $\beta$ to have 
\bq\label{estF:C3}
\| F(\eta, \ka)\|_{C^{3, \beta}}\le  M_2\left(\|\eta \|_{C^{1, \beta}}+|\ka|\| \varphi\|_{C^{1, \beta}}\right)\left(\| \eta\|_{C^{1, \beta}}+|\ka|| \varphi\|_{C^{1, \beta}}\right).
\eq
But $\| F(\eta, \ka)\|_{C^{3, \beta}}=\|\eta\|_{C^{3, \beta}}$ is unbounded in $\mathcal{C}$ owing to \eqref{unb:C1a}. Therefore, \eqref{unb:1} follows from \eqref{estF:C3} and the boundedness of $\ka$ in $\mathcal{C}$.
\end{proof}
Next, we prove that for dimension $d=1, 2$, \eqref{unb:1} can be further strengthened to the blowup of the maximum gradient.
\begin{lemm}\label{lemm:unbC1}
For $d=1, 2$ we have
\bq\label{unb:C1}
\sup_{\eta \in \Pi_1\mathcal{C}} \| \eta\|_{C^1(\T^d)}=\infty.
\eq
\end{lemm}
The proof of Lemma \ref{lemm:unbC1} relies on the following sharp invertibility  of the Dirichlet-to-Neumann operator in Lipschitz domains. 
\begin{theo}\label{theo:DN:Lp}
Let  $\eta \in W^{1, \infty}(\T^d)$ with $\| \eta\|_{W^{1, \infty}(\T^d)}\le m<\infty$ and $d\ge 1$.  There exists $\eps_*$ depending only on $(m, d, b, c^0)$ such that the following holds. For any $p\in (1 , 2+\eps^*)$, $G[\eta]: \mathring{W}^{1, p}(\T^d)\to \mathring{L}^p(\T^d)$ is invertible and there exists $C>0$ depending only on $(m, d, b, p, c^0)$ such that 
\bq
\|(G[\eta])^{-1}\|_{L^p(\T^d)\to W^{1, p}(\T^d)}\le C.
\eq
\end{theo}
Theorem \ref{theo:DN:Lp} is a direct consequence of Theorems 3.8 and 2.13 and in \cite{DahlbergKenig} regarding the optimal solvability of the Dirichlet and Neumann problem in Lipschitz domains with $W^{1, p}$ and $L^p$ boundary condition, respectively. 

{\it Proof of Lemma \ref{lemm:unbC1}}

Assume for contradiction that \eqref{unb:C1} is false, i.e.  $m:=\sup_{\eta \in \Pi_1\mathcal{C}} \| \eta\|_{C^1} <\infty$. By virtue of Theorem \ref{theo:DN:Lp}, there exist $p=p(m, d, b, c^0)>2$ and  $C=C(m, d, b, p, c^0)>0$ such that 
\bq
\| (G[\eta])^{-1}\p_1\eta\|_{W^{1, p}}\le C\| \p_1\eta\|_{L^p}\le C\| \eta\|_{W^{1, p}}.
\eq
For $d=1, 2$, Morrey's inequality implies $W^{1, p}(\T^d)\subset C^{0, \mu}(\T^d)$ for some $\mu=\mu(d, p)>0$, and hence 
\bq\label{sharpest:DN}
\| (G[\eta])^{-1}\p_1\eta\|_{C^{0, \mu}}\le C\| \eta\|_{W^{1, p}},\quad C=C(m, d, b, c^0).
\eq
 Consequently, we can apply \eqref{continverse:H} and \eqref{sharpest:DN} to have 
\[
\begin{aligned}
\| F(\eta, \ka)\|_{C^{2, \mu}}&\le M_1\left(\| \gamma (G[\eta])^{-1}\p_1\eta\|_{C^{0, \mu}}+|\ka|\| \varphi\|_{C^{0, \mu}} \right)\left(\| \gamma (G[\eta])^{-1}\p_1\eta\|_{C^{0, \mu}}+|\ka|\| \varphi\|_{C^{0, \mu}}\right)\\
&\le M_2(\| \eta\|_{W^{1, p}}+|\ka|\| \varphi\|_{C^{0, \mu}})(\| \eta\|_{W^{1, p}}+|\ka|\| \varphi\|_{C^{0, \mu}})\\
&\le M_3(\| \eta\|_{C^1}+|\ka|\| \varphi\|_{C^{0, \mu}})(\| \eta\|_{C^1}+|\ka|\| \varphi\|_{C^{0, \mu}}),
\end{aligned}
\]
where $M_1$ depends only on $(\sigma, g, d, b, c^0)$, and $M_2$ and $M_3$ depend only on $(m, \sigma, g, d, b,  c^0, |\gamma|)$. Combining this with  the boundedness of $\ka$ and $\| \eta\|_{C^1}$ in $\mathcal{C}$, we deduce that $\| F(\eta, \ka)\|_{C^{2, \mu}}$ is bounded for $(\eta, \ka)\in \mathcal{C}$. This contradicts the fact that  $F(\eta, \ka)=-\eta$ and $\| \eta\|_{C^{2, \mu}}$ is unbounded in $\mathcal{C}$ by virtue of \eqref{unb:1}. $\square$

{\bf Case 2:} $\ka$ is unbounded in $\mathcal{C}$.  We shall prove that $\| \eta\|_{C^1(\T^d)}$ is  unbounded in $\mathcal{C}$ for any $d\ge 1$. 
To obtain this we first note that  if $h\in \rC^{0, \beta}(\T^d)$, then  $f:=(g H+g I)^{-1}h\in \rC^{2, \beta}(\T^d)$ and
\bq\label{est:inverseH:Sobolev}
 \max\{\sigma, g\} \| f\|_{H^1}\ge \| h\|_{H^{-1}}
\eq
since  
 \[
\| h\|_{H^{-1}}=\| \sigma H(f)+gf\|_{H^{-1}}\le \sigma \left\| \frac{\na f}{\sqrt{1+|\na f|^2}}\right\|_{L^2}+g\| f\|_{H^{-1}}\le \sigma \| \na f\|_{L^2}+g\| f\|_{L^2}.
\]
 Using \eqref{est:inverseH:Sobolev} and the definition \eqref{def:F} of $F(\eta, \ka)$, we deduce 
\begin{align*}
\max\{\sigma, g\}\| \eta\|_{H^1}&=\max\{\sigma, g\}\| F(\eta, \ka)\|_{H^1}\\
&\ge \|-\gamma(G[\eta])^{-1} \p_1\eta+\ka \varphi \|_{H^{-1}}\\
&\ge |\ka|\| \varphi \|_{H^{-1}}-|\gamma|\|(G[\eta])^{-1} \p_1\eta\|_{H^{-1}}.
\end{align*}
On the other hand, it follows from \eqref{inverseDN:low} that 
\[
\|(G[\eta])^{-1} \p_1\eta\|_{H^{-1}}\le \|(G[\eta])^{-1} \p_1\eta\|_{H^{\mez}}\le M_1(\|\eta\|_{W^{1, \infty}})\| \eta\|_{H^{\mez}},
\]
where $M_1$ depends only on  $(d, b, c^0)$.  Consequently 
\[
|\ka|\| \varphi \|_{H^{-1}}\le \max\{\sigma, g\}\| \eta\|_{H^1}+|\gamma|M_1(\|\eta\|_{W^{1, \infty}})\| \eta\|_{H^{\mez}}\le M_2(\|\eta\|_{C^1})\| \eta\|_{C^1},
\]
where $M_2$ depends only on  $(d, b, c^0, \sigma, g, |\gamma|)$.  Since  $\ka$ in unbounded in $\mathcal{C}$, this implies that $\|\eta\|_{C^1}$ must be unboundedness in $\mathcal{C}$.

We have proven that in both Case 1 and Case 2,  the connected set $\mathcal{C}$ contains traveling waves that are  unboundedly large in $C^1$ for $d=1, 2$ and in $C^{1, \beta}$ for  $d\ge 3$. 

\subsubsection{Infinite depth} In this case we apply Theorem \ref{GIFT} with $U=X\times \Rr$ and there are only two alternatives, (i) and (ii). Lemma \ref{lemm:rigidity} again allows us to exclude (ii), so that (i) holds, i.e. $\mathcal{C}$ is unbounded. The rest of the proof follows  as in the finite depth case.

The proof of Theorem \ref{theo:main} is complete.

\appendix
\section{Fourier multipliers in H\"older spaces}\label{appendix}
\begin{defi}
Let $\chi:\Rr^d\to \Rr$ be a $C^\infty$ function satisfying $\chi(\xi)=1$ for $|\xi|\le 1/2$ and $\chi(\xi)=0$ for $|\xi|\ge 1$. Set $\tt(\xi)=\chi(\xi/2)-\chi(\xi)$ and
\[
\tt_j(\xi)=\tt\left(\frac{\xi}{2^j}\right),\quad j\ge 0,
\]
 so that 
 \[
 \chi(\xi)+\sum_{j=0}^\infty \tt_j(\xi)=1\quad\forall \xi \in \Rr^d.
 \]
Note that $\tt_j$ is supported in the annulus $\{2^{j-1}<|\xi|<2^{j+1}\}$. We then define the Fourier multipliers 
\[
\Delta_{-1}(D)=\chi(D),\quad \Delta_j(D)=\tt_j(D)~\text{for } j\ge 0,
\]
and obtain the Littlewood-Payley decomposition of identity
\[
I=\sum_{j=-1}^\infty \Delta_j.
\]
 \end{defi}
 \begin{defi}
For $s\in \Rr$, the Zygmund space $C^s_*(\T^d)$ is the space of distributions on $\T^d$ such that the norm
\bq
\| u\|_{C^s_*(\T^d)}=\sup_{j\ge -1}2^{sj}\| \Delta_j u\|_{L^\infty(\T^d)}
\eq
is finite. The space of $C^s_*$ distributions with mean zero is denoted by 
\bq
\rC^s_*(\T^d)=\left\{u\in C^s_*(\T^d): \langle u, 1\rangle=0\right\}\equiv \left\{u\in C^s_*(\T^d): \Delta_{-1}u=0\right\}.
\eq
\end{defi}
\begin{prop}\label{prop:multiplier}
Let $a:\Rr^d\setminus \{0\}\to \Cc$ be a smooth function such that for some $\tau\in \Rr$, 
\bq\label{multiplier}
\forall \beta \in \mathbb{N}^d,~\exists C_\beta>0,~\forall |\xi|> \mez,~|\p^\beta a(\xi)|\le C_\beta |\xi|^{\tau-|\beta|}.
\eq
Then the Fourier multiplier $a(D)$ is continuous from $\rC^s_*(\T^d)$ to  $\rC^{s-\tau}_*(\T^d)$ for all $s\in \Rr$. 
\end{prop}
\begin{proof}
This follows from the estimate 
\bq\label{est:Deltaj}
\| \Delta_ja(D)u\|_{L^\infty(\T^d)}\le C2^{sj}\| \Delta_j u\|_{L^\infty(\T^d)},\quad j\ge 0.
\eq
The proof of the $\Rr^d$ version of \eqref{est:Deltaj} can be found in \cite[Lemma 2.2]{BCD}. For $\T^d$ \eqref{est:Deltaj} can be proven similarly upon using the Poisson summation formula
\bq\label{Poissonsum}
\cF^{-1}(f)(x)=\sum_{k\in \Zz^d}\cF_{\Rr^d}^{-1}(f)(x+2\pi k)
\eq
for any Schwartz function $f:\Rr^d\to \Cc$. Here, we have denoted 
\[
\cF^{-1}(f)(x)=(2\pi)^{-d}\sum_{k\in \Zz^d}f(k)e^{ik\cdot x},\quad\cF^{-1}_{\Rr^d}(f)(x)=(2\pi)^{-d}\int_{\Rr^d}f(\xi)e^{i\xi\cdot x}d\xi.
\]
\end{proof}
Finally, we have following equivalence between Zygmumd and H\"older spaces.
\begin{prop}\label{prop:ZH}
For $k\in \Nn$ and $\a\in (0, 1)$, the spaces $\rC^{k, \a}(\T^d)$ and $\rC^{k+\a}_*(\T^d)$ are equal with equivalent norms. 
\end{prop}
Proposition \ref{prop:ZH} is a well-known fact  on $\Rr^d$ and can be similarly proven for $\T^d$ upon using the Poisson summation formula \eqref{Poissonsum}. Details can be found  in \cite[Proposition A.1]{GNP}.

\vspace{.1in}
{\noindent{\bf{Acknowledgment.}} 
The work of HQN was partially supported by NSF grant DMS-2205710. The author thanks I. Tice for various discussions on viscous surface waves. The author thanks W. Strauss for explaining the use of the Global Implicit Function Theorem in his joint work \cite{StraussWu}. The discussion has motivated this paper, which is dedicated to W. Strauss on the occasion of his 85th birthday. Finally,  the author thanks the referee for  constructive comments which have helped improve the presentation and accuracy  at places. 
}


\begin{thebibliography}{10}

\bibitem{APW} S. Agrawal, N. Patel, S. Wu.  Rigidity of acute angled corners for one phase Muskat interfaces. {\it Adv. Math.} 412 (2023), Paper No. 108801, 71 pp.

\bibitem{AmickToland} C. J. Amick, J. F. Toland. On solitary water-waves of finite amplitude. {\it Arch. Rational Mech. Anal.}, 76(1):9--95, 1981.
 
 \bibitem{AFT} C. J. Amick, L. E. Fraenkel, J. F. Toland. On the Stokes conjecture for the wave of extreme form. {\it Acta Math.},
148:193--214, 1982.
 
 \bibitem{ABZ3} T. Alazard, N. Burq, and C. Zuily.  On the Cauchy problem for gravity water waves. {\it Invent. Math.},  198 (2014), no. 1, 71--163.

\bibitem{AN} T.  Alazard, Q.-H. Nguyen. Endpoint Sobolev theory for the Muskat equation. {\it Comm. Math. Phys.} 397 (2023), no.3, 1043--1102.

\bibitem{Amb} D. M. Ambrose, Well-posedness of two-phase Darcy flow in 3D. {\it Q. Appl. Math.} 65(1), 189--203, 2007.

\bibitem{Bear} J. Bear. {\it Dynamics of fluids in porous media.} Dover Publications, 1988.


\bibitem{BCD} H. Bahouri, J-Y Chemin, and R. Danchin. {\it Fourier analysis and nonlinear partial differential equations}, volume
343 of Grundlehren der Mathematischen Wissenschaften [Fundamental Principles of Mathematical Sciences].
Springer, Heidelberg, 2011.

\bibitem{BoyerFabrie}F. Boyer, P. Fabrie. {\it Mathematical tools for the study of the incompressible Navier-Stokes equations and related models.} Applied Mathematical Sciences, 183. Springer, New York, 2013. xiv+525 pp.

\bibitem{CheGraShk} C.H. Cheng, R. Granero-Belinch\'on and S. Shkoller, Well-posedness of the Muskat problem with $H^2$ initial data. {\it Adv. Math.} 286, 32--104, 2016.

\bibitem{CDAD} Y. Cho, J. D. Diorio, T. R. Akylas, and J. H. Duncan. Resonantly forced gravity–capillary lumps on deep water. Part
2. Theoretical model. {\it J. Math. Fluid Mech.}, 672:288--306, 2011

\bibitem{ConstantinStrauss} A. Constantin, W. Strauss. Exact steady periodic water waves with vorticity. {\it Comm. Pure Appl. Math.}, 57(4):481--
527, 2004.

\bibitem{CorCorGan} A. C\'ordoba, D. C\'ordoba, and F. Gancedo, Interface evolution: the Hele-Shaw and Muskat problems. {\it Ann. Math.} 173(1), 477--542, 2011.

\bibitem{CorCorGan2} A. C\'ordoba, D. C\'ordoba, and F. Gancedo. Porous media: the Muskat problem in three dimensions. {\it Anal. \& PDE}, 6(2):447--497, 2013.

\bibitem{DahlbergKenig} B. J. E. Dahlberg, C. E.  Kenig. Hardy spaces and the Neumann problem in Lp for Laplace's equation in Lipschitz domains. {\it Ann. of Math.} 125(3), 437--465, 1987.

\bibitem{GGSP} E. Garcia-Juarez, J. Gomez-Serrano, S. V. Haziot, B. Pausader. The Desingularization of Small Moving Corners for the Muskat Equation. Preprint 2023, arXiv:2305.05046.

\bibitem{DGN} H. Dong, F. Gancedo. H. Q. Nguyen. Global well-posedness for the one-phase Muskat problem.  {\it Comm. Pure Appl. Math.}, to appear. 

\bibitem{DGN2}  H. Dong, F. Gancedo. H. Q. Nguyen. Global well-posedness for the one-phase Muskat problem. Preprint 2023, arXiv:2103.02656.

\bibitem{GilTru} D. Gilbarg, N. S. Trudinger. {\it Elliptic partial differential equations of second order. Second edition}. Grundlehren der mathematischen Wissenschaften [Fundamental Principles of Mathematical Sciences], 224. Springer-Verlag, Berlin, 1983. xiii+513 pp. 

\bibitem{GNP}  F. Gancedo, H. Q. Nguyen, N. Patel. Well-posedness for SQG sharp fronts with unbounded curvature. {\it Math. Models Methods Appl. Sci.} (32) no. 13, 2551--2599, (2022).

\bibitem{Han} Q. Han.  {\it Nonlinear elliptic equations of the second order}. Grad. Stud. Math., 171, American Mathematical Society, Providence, RI, 2016. viii+368 pp.


\bibitem{Ebbandflow} S. Haziot, V. M.  Hur, W. A. Strauss, J. F. Toland, J. F, E. Wahl\'en, Erik, S.  Walsh, M. H. Wheeler. Traveling water waves—the ebb and flow of two centuries. {\it Quart. Appl. Math.} 80 (2022), no. 2, 317--401.
 
\bibitem{Kielhofer} H. Kielh\"ofer.  {\it Bifurcation theory. An introduction with applications to PDEs.}  Applied Mathematical Sciences, 156. Springer-Verlag, New York, 2004. viii+346 pp. 

\bibitem{KoganemaruTice} J. Koganemaru. I. Tice.  Traveling wave solutions to the inclined or periodic free boundary incompressible Navier-Stokes equations.  Preprint (2022), arXiv:2207.07702.

\bibitem{Krasovskii} J. P. Krasovski\u{i}. On the theory of steady-state waves of finite amplitude. {\it \u{Z}. Vy\u{c}isl. Mat i Mat. Fiz.}, 1:836--855, 1961.

\bibitem{Leoni} G. Leoni. {\it Giovanni A first course in Sobolev spaces.} Graduate Studies in Mathematics, 105. American Mathematical Society, Providence, RI, 2009. xvi+607 pp.

\bibitem{Levi} T. Levi-Civita.  Determination rigoureuse des ondes permanentes d’ampleur finie.  {\it Math. Ann.} 93 (1925), 264--314.

\bibitem{DCDA} J. D. Diorio, Y. Cho, J. H. Duncan, T. R. Akylas. Resonantly forced gravity–capillary lumps on deep water. Part
1. Experiments. {\it J. Math. Fluid Mech.} , 672:268--287, 2011.

\bibitem{LeoniTice} G. Leoni, I. Tice. Traveling wave solutions to the free boundary incompressible Navier-Stokes equations. {\it Comm. Pure Appl. Math.} (2022).

\bibitem{MasnadiDuncan} N. Masnadi,J. H. Duncan. The generation of gravity–capillary solitary waves by a pressure source moving at a
trans-critical speed. {\it J. Fluid Mech.}, 810:448--474, 2017.

\bibitem{Mat} B-V Matioc. The muskat problem in 2d: equivalence of formulations, well-posedness, and regularity
results. {\it Analysis \& PDE}, 12(2), 281-332, 2018.

\bibitem{Nekrasov} A. I.  Nekrasov. On steady waves. Izv. Ivanovo-Voznesensk. Politekhn. In-ta 3 (1921).

\bibitem{NgPa} Huy Q. Nguyen, B. Pausader. A paradifferential approach for well-posedness of the Muskat problem. {\it Arch. Ration. Mech. Anal.} 237 (2020), no. 1, 35--100.

\bibitem{Ng-st} Huy Q. Nguyen.  On well-posedness of the Muskat problem with surface tension. {\it Adv. Math.} 374 (2020), 107344, 35 pp. 

\bibitem{NguyenTice} H. Q. Nguyen, I. Tice. Traveling wave solutions to the one-phase Muskat problem: existence and stability. Arch. Ration. Mech. Anal. 248 (2024), no. 1, Paper No. 5, 58 pp.


\bibitem{ParkCho1} B. Park, Y. Cho. Experimental observation of gravity–capillary solitary waves generated by a moving air suction. {\it J.
Math. Fluid Mech.}, 808:168--188, 2016.


\bibitem{Plotnikov} P. I. Plotnikov. Proof of the Stokes conjecture in the theory of surface waves. {\it Stud. Appl. Math.}, 108(2):217--244, 2002.
Translated from {\it Dinamika Sploshn. Sredy} No. 57 (1982), 41--76.

\bibitem{StevensonTice1} N. Stevenson, I. Tice. Traveling wave solutions to the multilayer free boundary incompressible Navier-Stokes equations. \emph{SIAM J. Math. Anal.} \textbf{53} (2021), no. 6, 6370--6423. 

\bibitem{StevensonTice2} N. Stevenson, I. Tice. Well-posedness of the traveling wave problem for the free boundary compressible Navier-Stokes equations. Preprint (2023), arXiv:2301.00773.

\bibitem{StraussWu} W. A. Strauss, Y. Wu. Rapidly rotating stars. {\it Comm. Math. Phys.} 368 (2019), no. 2, 701--721.

\bibitem{StraussWu2} W. A. Strauss, Y. Wu. Rapidly rotating white dwarfs. {\it Nonlinearity} 33 (2020), no. 9, 4783--4798.

\bibitem{StraussWu3} W. A. Strauss, Y. Wu. Global continuation of a Vlasov model of rotating galaxies. {\it Kinet. Relat. Models} 16 (2023), no. 4, 605--623.

\bibitem{Struik} D. J. Struik. Determination rigoureuse des ondes irrotationnelles periodiques dans un canal a profondeur finie. {\it Math. Ann.} 95 (1926), 595--634.

\bibitem{Stokes} G. G. Stokes. On the theory of oscillatory waves. {\it Trans. Cambridge Philos. Soc.} 8 (1847), 441--455.

\bibitem{Toland} J. F. Toland. On the existence of a wave of greatest height and Stokes’s conjecture. {\it Proc. Roy. Soc. London Ser. A},
363(1715):469--485, 1978.


 \end{thebibliography}
\end{document}